\documentclass{article}
\pdfoutput=1 

\usepackage{enumerate}
\usepackage{hyperref}
\usepackage{latexsym,amsmath,amssymb,stmaryrd,amsthm}
\usepackage{color}
\usepackage[all]{xy}
\usepackage[inline]{enumitem}
\usepackage{macros}

\begin{document}





\title{Products and Projective Limits of Continuous Valuations on $T_0$ Spaces}
\author{Jean Goubault-Larrecq\\
  LSV, ENS Paris-Saclay, CNRS, Universit\'e Paris-Saclay \\
  94230 Cachan, France \\
  \url{goubault@lsv.fr}}



\maketitle

\begin{abstract}
  \noindent We show analogues of the Daniell-Kolmogorov and Prohorov
  theorems on the existence of projective limits of measures, in the
  setting of continuous valuations on $T_0$ topological spaces.
\end{abstract}

\section{Introduction}
\label{sec:intro}

Consider the following problem.  We are given an arbitrary family
${(X_i)}_{i \in I}$ of topological spaces, and for each finite subset
$J$ of $I$, a measure $\nu_J$ on $X_J \eqdef \prod_{i \in J} X_i$.
There are obvious projections $p_{JJ'} \colon X_{J'} \to X_J$
($J \subseteq J'$), and we require that the measures are
\emph{compatible} in the sense that $p_{JJ'} [\nu_{J'}] = \nu_J$ for
all $J \subseteq J'$ in $\Pfin (I)$, where $p_{JJ'} [\nu_{J'}]$ is the
image measure defined by
$p_{JJ'} [\nu_{J'}] (E) \eqdef \nu_{J'} (p_{JJ'}^{-1} (E))$.  Is there
a measure $\nu$ on the total space $\prod_{i \in I} X_i$ such that
$p_J [\nu] = \nu_J$ for every $J \in \Pfin (I)$, where
$p_J \colon X \to X_j$ is the obvious projection?

This question was solved long ago by Daniell \cite{Daniell:prod} and
Kolmogorov \cite{Kolmo:proba}, and the solution was gradually
generalized to: it does, provided each $X_i$ is Hausdorff and
considered with its Borel $\sigma$-algebra, the measures $\nu_J$ and
$\nu$ are taken to be defined on the product $\sigma$-algebras (rather
than on the Borel $\sigma$-algebras of the topological products, which
are larger), and every $\nu_J$ is tight.

The resulting Daniell-Kolmogorov theorem is of paramount importance in
mathematics.  Together with the existence of finite products of
$\sigma$-finite measures, this yields the existence of arbitrary
products of tight $\sigma$-finite measures.  The theorem is also the
basis of the theory of stochastic processes, where $I$ is $\real$,
including L\'evy processes (of which Wiener measure is a particular
case).  In general, it asserts the existence of measures on sets of
\emph{execution paths} of a stochastic transition system, with
discrete or continuous-time dynamics.

The first contribution of the present paper is a variant of that
theorem where measures are replaced by \emph{continuous valuations}
\cite{jones89,Jones:proba}.  The gain is that the spaces $X_i$ are no
longer required to be Hausdorff, and are arbitrary topological spaces,
and that the valuations $\nu_i$ need only be continuous, not tight.

Continuous valuations were proposed by Jones and Plotkin
\cite{jones89,Jones:proba} as a convenient alternative to measures in
the setting of programming language semantics, following
Saheb-Djahromi \cite{saheb-djahromi:meas}.  They are very close to
measures, and one may consult \cite{KL:measureext} to understand the
relation between continuous valuations and measures.

The typical proofs of the Daniell-Kolmogorov theorem go through
theorems establishing the existence of projective limits of measures.
The key theorem here is due to Prohorov \cite{Prohorov:projlim}, who
showed that projective limits of bounded measures exist under a
so-called uniform tightness assumption.  A close theorem by Bochner
\cite{Bochner:harmonic} involves a so-called sequential maximality
assumption.

Our proof of our variant of the Daniel-Kolmogorov theorem also uses
projective limits, but only of a very particular kind where existence
is obvious, and based on systems of embedding-projection pairs,
which should be familiar to domain theorists.

We will nonetheless address the question of the existence of general
projective limits of continuous valuations, which are of interest
beyond the case of the Daniell-Kolmogorov theorem.  Since our aim is
to work on classes of $T_0$ spaces, not just Hausdorff spaces, this
will require us to propose a reworked definition of tightness and of
uniform tightness.

\paragraph{Outline.}  We give a few preliminary definitions and
results in Section~\ref{sec:preliminaries}.  Our central problem,
Problem~\ref{pb:prohorov}, is stated there.  For now, let us call it
the question of existence of projective limits of continuous
valuations.  In Section~\ref{sec:uniq-proj-limits}, we show that
projective limits of continuous valuations are unique, if they exist
at all.  We promised we would look at a simple case of projective
limits, of projective systems consisting of embedding-projection
pairs: this is the topic of Section~\ref{sec:ep-systems}.  Together
with an adequate notion of support of continuous valuations, this
allows us to derive the announced analogue of the Daniell-Kolmogorov
theorem in Section~\ref{sec:dani-kolm-theor}.  This concludes the first
part of the paper.

In a second part, we solve the general problem of existence of
projective limits of continuous valuations, following Prohorov's
approach, suitably generalized to $T_0$ spaces.  This notably requires
a new definition of tightness and uniform tightness, which we give and
develop in Section~\ref{sec:tightn-unif-tightn}.  We shall see that
all continuous valuations are tight on any consonant space, a pretty
large class of spaces that includes all locally compact spaces and all
complete metric spaces.  Then, the existence of a \emph{tight}
projective limit of continuous valuations will be equivalent to
uniform tightness.  The construction will be pretty transparent.

Checking uniform tightness is difficult in general, though.  We shall
give two cases where it is automatically satisfied.  Both cases rely
on a theorem on projective limits of compact sober spaces, which
Fujiwara and Kato call \emph{Steenrod's theorem}
\cite{FK:rigid:geometry}, and which we shall state in
Section~\ref{sec:steendrods-theorem}.  Both cases also apply to
projective systems indexed by a set that has a countable cofinal
subset, so as to avoid certain pathologies (see
Section~\ref{sec:preliminaries}).  We shall then see that, under that
assumption, projective limits of continuous valuations exist provided
all the spaces in the given projective system are locally compact and
sober (Section~\ref{sec:proj-syst-core}), or provided they are
$G_\delta$ subspaces of locally compact sober spaces and all the given
valuations are locally finite (Section~\ref{sec:proj-limits-compl}).
Beyond locally compact sober spaces, the latter case includes all
complete metric spaces (including all Polish spaces), and all
continuous complete quasi-metric spaces in their $d$-Scott topology.

\section{Preliminaries}
\label{sec:preliminaries}

We use \cite{JGL-topology} as our main reference on topology and
domain theory, as well as \cite{GHKLMS:contlatt} on domain theory.

We always write $\leq$ for the specialization preordering of
topological spaces.  In particular, every open set if upwards closed
and every continuous map is monotonic.  We write $\upc A$ for the
upward closure of a set $A$, $\dc A$ for its downward closure.  The
specialization preordering on a subspace $A$ of a space $X$ is the
restriction of the specialization preordering of $X$ to $A$.

Given a general preordered set $X$ (for example a topological space),
a \emph{monotone net} is a family ${(x_i)}_{i \in I, \sqsubseteq}$ of
points of $X$ indexed by a directed preordered set $I, \sqsubseteq$,
such that $i \sqsubseteq j$ implies $x_i \leq x_j$.  A directed family
is always the underlying set of elements of a monotone net.

A \emph{compact} subset $K$ of a space $X$ is one with the Heine-Borel
property, equivalently: $K$ is compact if and only if, for every
directed family ${(U_i)}_{i \in I}$ of open subsets of $X$, if
$K \subseteq \dcup_{i \in I} U_i$ then $K \subseteq U_i$ for some
$i \in I$.  (Superscript arrows indicate directness.)  In particular,
compactness does not require Hausdorffness.

A \emph{saturated} subset $A$ of a space $X$ is one that is the
intersection of its open neighborhoods.  The saturated subsets are
exactly the upwards closed subsets, relatively to the specialization
preordering.

A \emph{locally compact} space $X$ is a space where every open subset
$U$ is the directed union of the interiors $\interior Q$ of compact
saturated subsets $Q$ of $U$.

A \emph{sober} space $X$ is a $T_0$ space in which every irreducible
closed subset is the closure $\dc x$ of some (unique) point $x$.  A
closed set $C$ is \emph{irreducible} if it is non-empty, and satisfies
the property that if $C$ is included in the union of two closed sets,
then it must be included in one of them.

A \emph{well-filtered} space $X$ is a space such that every open
subset $U$ that contains a filtered intersection $\fcap_{i \in I} Q_i$
of compact saturated subsets $Q_i$ must contain some $Q_i$ already.
Every sober space is well-filtered.  In a well-filtered space, every
filtered intersection $\fcap_{i \in I} Q_i$ of compact saturated
subsets is compact saturated.

Let $\creal \eqdef \Rp \cup \{\infty\}$ be the set of extended
non-negative real numbers.  We write $\Open X$ for the lattice of open
subsets of a topological space $X$.

A \emph{valuation} $\nu$ on a topological space $X$ is a map from
$\Open X$ to $\creal$ that satisfies: $\nu (\emptyset) = 0$
(strictness); $U \subseteq V$ implies $\nu (U) \leq \nu (V)$
(monotonicity); and
$\nu (U \cup V) + \nu (U \cap V) = \nu (U) + \nu (V)$ (modularity).

A \emph{continuous valuation} is a valuation that is Scott-continuous,
i.e., $\nu (\dcup_{i \in I} U_i) = \dsup_{i \in I} \nu (U_i)$.  A
valuation $\nu$ is \emph{bounded} if and only if $\nu (X) < \infty$.

A \emph{projective system} in a category $\catc$ is a functor from
$(I, \sqsubseteq)^{op}$ to $\catc$, for some directed preordered set
$I, \sqsubseteq$.  Explicitly, this is a family of objects
${(X_i)}_{i \in I, \sqsubseteq}$ of $\catc$, where $I$ is a set with a
preorder $\sqsubseteq$ that makes $I$ directed; together with
morphisms $p_{ij} \colon X_j \to X_i$ for all $i \sqsubseteq j$ in
$I$, satisfying $p_{ii} = \identity {X_i}$ and
$p_{ij} \circ p_{jk} = p_{ik}$ for all $i \sqsubseteq j \sqsubseteq k$
in $I$.

The maps $p_{ij}$ will familiarly be called the \emph{bonding maps}
of the projective system.

We shall write such a projective system
${(p_{ij} \colon X_j \to X_i)}_{i \sqsubseteq j \in I}$, for short.  A
limit of such a system is a universal cone
$X, {(p_i \colon X \to X_i)}_{i \in I}$, and is familiarly called a
\emph{projective limit} of the system.

Projective limits, if they exist, are unique up to isomorphism.  In
$\Topcat$, all projective limits exist.  There is a \emph{canonical
  projective limit} in $\Topcat$, which is described as follows: $X$
is the subspace of those
$\vec x \eqdef {(x_i)}_{i \in I} \in \prod_{i \in I} X_i$ such that
$p_{ij} (x_j) = x_i$ for all $i \sqsubseteq j$ in $I$---notably, the
topology on $X$ is the subspace topology induced by the product
topology on $\prod_{i \in I} X_i$; and $p_i$ just maps $\vec x$ to
$x_i$.

One should be aware that projective limits can behave in a somewhat
odd way.  Henkin had shown that, even if every $X_i$ is non-empty and
every $p_{ij}$ is surjective, the projective limit $X$ can be empty
\cite{Henkin:invlimit}, and the maps $p_i$ are therefore not
surjective.  Waterhouse \cite{Waterhouse:projlim:empty} gave an
elementary example: take an uncountable set $A$, let $I$ be the family
of finite subsets of $A$ ordered by inclusion, let $X_i$ be the set
(a.k.a., the discrete space, if you wish) of all injective maps from
$i \subseteq A$ to $\nat$, and $p_{ij}$ be defined by restriction.
One sees that the projective limit is isomorphic to the set of
injection of $A$ into $\nat$, which is empty.

Such pathologies do not happen if $I$ has a countable cofinal
subfamily.

Our interest on projective limits lies within the category
$\Val\Topcat$ of \emph{valued spaces} (our analogue of measure
spaces), whose objects are pairs $(X, \nu)$ of a topological space $X$
and a continuous valuation $\nu$ on $X$, and whose morphisms
$f \colon (X, \mu) \to (Y, \nu)$ are the continuous maps
$f \colon X \to Y$ such that $f [\mu] = \nu$.
The notation $f [\mu]$ denotes the \emph{image valuation} of $\mu$ by
$f$, defined by $f [\mu] (V) \eqdef \mu (f^{-1} (V))$ for every $V \in
\Open Y$.  We shall sometimes use the same notation $f [\mu]$ for
arbitrary maps $\mu \colon \Open X \to \creal$.

One of the central problems that we attack in this paper is:
\begin{problem}[Projective limits of continuous valuations]
  \label{pb:prohorov}
  Consider a projective system
  ${(p_{ij} \colon (X_j, \nu_j) \to (X_i, \nu_i))}_{i \sqsubseteq j
    \in I}$ of valued spaces, and let $X, {(p_i)}_{i \in I}$ be a
  projective limit of the underlying projective system
  ${(p_{ij} \colon X_j \to X_i)}_{i \sqsubseteq j \in I}$ of
  topological spaces.

  Find sufficient, general conditions that ensure the existence of a
  continuous valuation $\nu$ on $X$ such that $p_i [\nu] = \nu_i$ for
  every $i \in I$.
\end{problem}
Note how relevant Henkin's and Waterhouse's examples of empty
projective limits are.  If $\nu$ exists as required in
Problem~\ref{pb:prohorov}, then $\nu (X) = \nu_i (X_i)$ for every
$i \in I$.  If the projective limit $X$ is empty, then $\nu (X)$ must
be equal to $0$, and therefore Problem~\ref{pb:prohorov} has a
solution if and only if each $\nu_i$ is the zero valuation.  This
partly justifies why we will assume that $I$ has a countable cofinal
subset in the final two sections of this paper.

\section{Uniqueness of projective limits}
\label{sec:uniq-proj-limits}

We first show that if Problem~\ref{pb:prohorov} has a solution, it
must be unique.

\begin{lemma}
  \label{lemma:prohorov:top}
  Let ${(p_{ij} \colon X_j \to X_i)}_{i \sqsubseteq j \in I}$ be a
  projective system of topological spaces, and $X, {(p_i)}_{i \in I}$
  be some projective limit of that system.

  For every open subset $U$ of $X$, for every $i \in I$, there is a
  largest open subset $U_i$ of $X_i$ such that $p_i^{-1} (U_i)
  \subseteq U$.  We write it $p_i^* (U)$.  Then:
  \begin{enumerate}
  \item for all $i \sqsubseteq j$ in $I$, $p_{ij}^{-1} (p_i^* (U))
    \subseteq p_j^* (U)$;
  \item for all $i \sqsubseteq j$ in $I$, $p_i^{-1} (p_i^* (U)) \subseteq
    p_j^{-1} (p_j^* (U))$;
  \item $U = \dcup_{i \in I} p_i^{-1} (p_i^* (U))$.
  \end{enumerate}
\end{lemma}
\begin{proof}
  It does not matter whether one takes one projective limit or
  another, since they are all isomorphic.  Hence we consider that $X$
  is the canonical projective limit
  $\{\vec x \in \prod_{i \in I} X_i \mid \forall i \sqsubseteq j, x_i
  = p_{ij} (x_j)\}$, and that $p_i$ is projection onto the $i$th
  coordinate.

  The union $U_i$ of all the open subsets of $X_i$ whose inverse image
  by $p_i$ is included in $U$ is also such that
  $p_i^{-1} (U_i) \subseteq U$.  Hence $U_i$ is the largest open
  subset such that $p_i^{-1} (U_i) \subseteq U$.  This shows that
  $p_i^* (U) \eqdef U_i$ is well defined.

  1. If $i \sqsubseteq j$, then
  $p_i^{-1} (p_i^* (U)) = p_j^{-1} (p_{ij}^{-1} (p_i^* (U)))$, since
  $p_i = p_{ij} \circ p_j$, and that is included in $U$.  Since $p_j^* (U)$
  is the largest open subset $U_j$ such that
  $p_j^{-1} (U_j) \subseteq U$, $p_{ij}^{-1} (p_i^* (U))$ must be
  included in $p_j^* (U)$.

  2. We again use the equality
  $p_i^{-1} (p_i^* (U)) = p_j^{-1} (p_{ij}^{-1} (p_i^* (U)))$.  Now we
  use item~1 to deduce that
  $p_j^{-1} (p_{ij}^{-1} (p_i^* (U))) \subseteq p_j^{-1} (p_j^* (U))$.

  3. By item~2, the family ${(p_i^{-1} (p_i^* (U)))}_{i \in I}$ is
  directed.  By definition, $p_i^{-1} (p_i^* (U)) \subseteq U$ for
  every $i \in I$.  It remains to show that $U$ is included in
  $\dcup_{i \in I} p_i^{-1} (p_i^* (U))$.

  Since $X$ has the subspace topology induced by the inclusion in
  $\prod_{i \in I} X_i$, and $p_i$ is projection on the $i$th
  component, the topology of $X$ is generated by subbasic open sets
  $p_i^{-1} (V_i)$, with $V_i$ open in $X_i$.  Call $\mathcal B$ the
  family of all those subbasic open sets.

  The whole set $X$ is equal to $p_i^{-1} (X_i)$, for any $i \in I$,
  hence is in $\mathcal B$.

  Consider any two subbasic open sets $p_i^{-1} (V_i)$ and
  $p_i^{-1} (V_j)$.  Since $I$ is directed, there is a $k \in I$ above
  both $i$ and $j$.  We claim that
  $p_i^{-1} (V_i) \cap p_j^{-1} (V_j) = p_k^{-1} (p_{ik}^{-1} (V_i)
  \cap p_{jk}^{-1} (V_j))$.  Every element $\vec x$ in the left-hand
  side is such that $x_i \in V_i$ and $x_j \in V_j$.  Since
  $x_i = p_{ik} (x_k)$, $x_k$ is in $p_{ik}^{-1} (V_i)$ and similarly
  $x_k$ is in $p_{jk}^{-1} (V_j)$.  This shows that $\vec x$ is in the
  right-hand side.  Conversely, any element $\vec x$ of
  $p_k^{-1} (p_{ik}^{-1} (V_i) \cap p_{jk}^{-1} (V_j))$ is such that
  $x_k$ is both in $p_{ik}^{-1} (V_i)$ and in $p_{jk}^{-1} (V_j)$, so
  $x_i \in V_i$ and $x_j \in V_j$.

  It follows that $\mathcal B$ is closed under finite intersections.
  We can therefore write $U$ as a union of elements of $\mathcal B$.
  For every $\vec x \in U$, let $p_i^{-1} (U_i)$ be an element of
  $\mathcal B$ included in $U$ that contains $\vec x$.  Since
  $p_i^{-1} (U_i) \subseteq U$, $U_i$ is included in $p_i^*  (U_i)$,
  and therefore
  $\vec x \in p_i^{-1} (U_i) \subseteq p_i^{-1} (p_i^* (U))$, so that
  $\vec x$ is in $\dcup_{i \in I} p_i^{-1} (p_i^*  (U))$.
\end{proof}

\begin{remark}
  \label{rem:p*}
  A more synthetic description of $p_i^*$ is as a right-adjoint to
  the map $p_i^{-1} \colon \Open {X_i} \to \Open X$.  Indeed $p_i^{-1}
  (U_i) \subseteq U$ if and only if $U_i \subseteq p_i^* (U)$.
\end{remark}

\begin{proposition}
  \label{prop:prohorov:unique}
  Let ${(p_{ij} \colon X_j \to X_i)}_{i \sqsubseteq j \in I}$ be a
  projective system of topological spaces, and $X, {(p_i)}_{i \in I}$
  be some projective limit of that system.  Given any family of
  Scott-continuous maps $\nu_i \colon \Open {X_i} \to \creal$,
  $i \in I$, such that for all $i \sqsubseteq j$ in $I$,
  $\nu_i = p_{ij} [\nu_j]$, there is at most one Scott-continuous map
  $\nu \colon \Open X \to \creal$ such that, for every $i \in I$,
  $\nu_i = p_i [\nu]$.

  If such a map $\nu$ exists, and every $\nu_i$ is a continuous
  valuation, then $\nu$ is a continuous valuation.
\end{proposition}
\begin{proof}
  If $\nu$ exists, then for every family of open subsets $U_i$ of
  $X_i$, $i \in I$, such that ${(p_i^{-1} (U_i))}_{i \in I}$ is
  directed,
  $\nu (\dcup_{i \in I} p_i^{-1} (U_i)) = \dsup_{i \in I} \nu
  (p_i^{-1} (U_i)) = \dsup_{i \in I} \nu_i (U_i)$, so $\nu$ is
  determined uniquely on those open subsets of the form
  $\dcup_{i \in I} p_i^{-1} (U_i)$.  Lemma~\ref{lemma:prohorov:top},
  item~3, assures us that every open subset $U$ of $X$ can be written
  in this way, by letting $U_i \eqdef p_i^{-1} (U_i)$.  Therefore
  $\nu$ is determined uniquely.

  Now assume that every $\nu_i$ is a continuous valuation, still
  assuming that $\nu$ exists.  Let
  $U \eqdef \dcup_{i \in I} p_i^{-1} (U_i)$ and
  $V \eqdef \dcup_{i \in I} p_i^{-1} (V_i)$ be two open subsets of
  $X$, where $U_i \eqdef p_i^* (U)$ and $V_i \eqdef p_i^* (V)$.
  $U \cup V$ is equal to $\dcup_{i \in I} p_i^{-1} (U_i \cup V_i)$.
  The intersection $U \cap V$ is also equal to
  $\dcup_{i \in I} p_i^{-1} (U_i \cap V_i)$, but that requires some
  checking.  For every element $\vec x$ of $U \cap V$, $\vec x$ is in
  $p_i^{-1} (U_i)$ for some $i \in I$ and in $p_j^{-1} (V_j)$ for some
  $j \in I$.  Since $I$ is directed, pick $k \in I$ above $i$ and $j$.
  Since the unions defining $U$ and $V$ are over monotone nets,
  $\vec x$ is also in $p_k^{-1} (U_k)$ and in $p_k^{-1} (V_k)$, hence
  in $p_k^{-1} (U_k \cap V_k)$, hence in
  $\dcup_{i \in I} p_i^{-1} (U_i \cap V_i)$.  The converse inclusion
  $\dcup_{i \in I} p_i^{-1} (U_i \cap V_i) \subseteq U \cap V$ is
  clear.

  Knowing that, $\nu (U \cup V) + \nu (U \cap V)$ is equal to
  $\dsup_{i \in I} \nu_i (U_i \cup V_i) + \dsup_{i \in I} \nu_i (U_i
  \cap V_i)$, hence to
  $\dsup_{i \in I} (\nu_i (U_i \cup V_i) + \nu_i (U_i \cap V_i))$ since
  addition is Scott-continuous.  In turn, that is equal to
  $\dsup_{i \in I} (\nu_i (U_i) + \nu_i (V_i)) = \dsup_{i \in I} \nu_i
  (U_i) + \dsup_{i \in I} \nu_i (V_i) = \nu (U) + \nu (V)$.
\end{proof}

\section{Ep-systems}
\label{sec:ep-systems}

An \emph{embedding-projection pair}, or \emph{ep-pair} for short, is a
pair of continuous maps
$\xymatrix{X \ar@<1ex>[r]^e & Y \ar@<1ex>[l]^p}$ such that
$p \circ e = \identity X$ and $e \circ p \leq \identity Y$.  (The
preorderings we consider are the specialization preorderings.)  Then
$p$ is called a \emph{projection} of $Y$ onto $X$, and $e$ is the
associated \emph{embedding}.

In general, we call a continuous map $p \colon Y \to X$ a
\emph{projection} if and only if it is the projection part of some
ep-pair.  If that is the case, and if $Y$ is $T_0$, then the
associated embedding $e \colon X \to Y$ must be such that, for every
$x \in X$, $e (x)$ is the smallest element $y \in Y$ such that
$x \leq p (y)$, so that $e$ is determined uniquely from $p$.

In a projective system of $T_0$ spaces $X_i$, $i \in I$, whose bonding
maps $p_{ij}$ are projections, we obtain corresponding embeddings
$e_{ij}$ this way, and since $e_{ij}$ is determined uniquely from
$p_{ij}$, it follows that $e_{ii}=\identity {X_i}$ for every
$i \in I$, and $e_{jk} \circ e_{ij} = e_{ik}$ for all
$i \sqsubseteq j \sqsubseteq k$ in $I$.

An \emph{ep-system} in $\catc$ is, categorically, a functor from
$(I, \sqsubseteq)^{op}$ to $\catc^{\text{ep}}$, where $I, \sqsubseteq$
is a directed preorder and $\catc^{\text{ep}}$ is the category whose
objects are the same as those of $\catc$, and whose morphisms are the
ep-pairs.  Explicitly, this is given by:
\begin{enumerate*}[label=(\roman*)]
\item a family of objects $X_i$ of $\catc$, $i \in I$;
\item ep-pairs $\xymatrix{X_i \ar@<1ex>[r]^{e_{ij}} & X_j
    \ar@<1ex>[l]^{p_{ij}}}$ for all $i \sqsubseteq j$ in $I$,
  satisfying:
\item $e_{ii} = p_{ii} = \identity {X_i}$ for every $i \in I$, and:
\item $p_{ij} \circ p_{jk} = p_{ik}$ and
  $e_{jk} \circ e_{ij} = e_{ik}$ for all
  $i \sqsubseteq j \sqsubseteq k$ in $I$.
\end{enumerate*}

Every ep-system has an underlying projective system, which we obtain
by forgetting the embeddings $e_{ij}$.  If every $X_i$ is $T_0$, we
have seen that we could reconstruct them anyway, under the promise
that every $p_{ij}$ is a projection at all.  Hence it makes sense to
talk about the projective limit of an ep-system.

Those are entirely classical notions, see \cite{AJ:domains} for
example.

The following says that projective limits of ep-systems are as nice as
they can be.  Since each $p_i$ is a projection, notably (item~2), all
of them are surjective, and in particular the projective limit cannot
be empty, unless every $X_i$ is empty.
\begin{lemma}
  \label{lemma:epsystem}
  Given an ep-system in $\Topcat$ consisting of spaces $X_i$,
  $i \in I$, and ep-pairs $(p_{ij}, e_{ij})$, $i \sqsubseteq j$ in
  $I$, and letting $X$ be the canonical projective limit $\{\vec x \in
  \prod_{i \in I} X_i \mid \forall i \sqsubseteq j, x_i = p_{ij}
  (x_j)\}$, define $p_i \colon X \to X_i$ by $p_i (\vec x) \eqdef
  x_i$.  Then:
  \begin{enumerate}
  \item every map $p_i$ is continuous;
  \item there are maps $e_i \colon X_i \to X$ that make $(p_i, e_i)$
    into ep-pairs, $i \in I$;
  \item for all $i \sqsubseteq j$ in $I$, $e_j \circ e_{ij} = e_i$;
  \item for every $\vec x \in X$, the family
    ${(e_i (p_i (\vec x)))}_{i \in I, \sqsubseteq}$ is a monotone net
    in $X$, and has $\vec x$ as supremum;
  \item for every open subset $U$ of $X$,
    ${((e_i \circ p_i)^{-1} (U))}_{i \in I, \sqsubseteq}$ is a
    monotone net in $\Open X$, and its union is $U$.
  \end{enumerate}
\end{lemma}
\begin{proof}
  1 is clear.

  2. For every $x \in X_i$, for every $j \in I$, we build the $j$th
  component $e_i (x)_j$ of $e_i (x)$ by first finding an index
  $k \in I$ such that $i, j \sqsubseteq k$, using the fact that $I$ is
  directed, and by letting $e_i (x)_j \eqdef p_{jk} (e_{ik} (x))$.
  Let us check that this is independent of the chosen $k$.  If we had
  chosen another $\ell \in I$ above $i$ and $j$, then we must show
  that $p_{jk} (e_{ik} (x)) = p_{j\ell} (e_{i\ell} (x))$.  Pick yet
  another element $m \in I$, this time above $k$ and $\ell$.  We have
  $p_{jm} (e_{im} (x)) = p_{jk} (p_{km} (e_{km} (e_{ik} (x))))$ by
  condition (iv) of ep-systems, and that is equal to
  $p_{jk} (e_{ik} (x))$ by the definition of ep-pairs.  By the same
  argument, with $\ell$ replacing $k$,
  $p_{jm} (e_{im} (x)) = p_{j\ell} (e_{i\ell} (x))$, so
  $p_{jk} (e_{ik} (x)) = p_{j\ell} (e_{i\ell} (x))$.

  In order to show that $e_i$ is continuous, it is enough to show that
  $x \mapsto e_i (x)_j$ is continuous for every $j \in I$.  Picking
  any $k \in I$ above $i$ and $j$, that amounts to the continuity of
  $p_{jk} \circ e_{ik}$.

  We compute $p_i (e_i (x))$.  This is $e_i (x)_i$, namely the case
  $j=i$ of the definition.  We take $k=i=j$, so that $e_i (x)_i =
  p_{ii} (e_{ii} (x)) = x$.

  By definition, for every $\vec x \in X$, for every $j \in I$, the
  $j$th component $e_i (p_i (\vec x))_j$ of $e_i (p_i (\vec x))$ is
  $p_{jk} (e_{ik} (x_i))$, where $k$ is any element of $I$ above $i$
  and $j$.  Since $\vec x$ is in $X$, $x_i$ is equal to
  $p_{ik} (x_k)$, so
  $e_i (p_i (\vec x))_j = p_{jk} (e_{ik} (p_{ik} (x_k)))$.  Since
  $e_{ik} \circ p_{ik} \leq \identity {X_k}$, and since $p_{jk}$,
  being continuous, is monotonic,
  $e_i (p_i (\vec x))_j \leq p_{jk} (x_k) = x_j$.  It follows that
  $e_i \circ p_i \leq \identity X$.

  3. Let $x \in X_i$, and $i \sqsubseteq j$.  For every $m \in I$, the
  $m$th component of $e_j (e_{ij} (x))$ is
  $p_{mk} (e_{jk} (e_{ij} (x)))$, where $k$ is any element of $I$
  above $m$ and $j$.  That is equal to $p_{mk} (e_{ik} (x))$, hence to
  $e_i (x)$.

  4. It suffices to show that, for every $m \in I$, the family
  ${(e_i (x_i)_m)}_{i \in I, \sqsubseteq}$ is a monotone net in $X_m$,
  and has $x_m$ as supremum.

  Assume $i \sqsubseteq j$.  We wish to show that
  $e_i (x_i)_m \leq e_j (x_j)_m$.  In order to do so, we pick some
  $k \in I$ above $j$, and $m$ (hence also above $i$).  Then
  $e_{ik} (x_i) = e_{jk} (e_{ij} (x_i))$.  Since $\vec x$ is in $X$,
  $x_i = p_{ij} (x_j)$, so $e_{ij} (x_i) \leq x_j$.  The map $e_{jk}$
  is continuous hence monotonic, so $e_{ik} (x_i) \leq e_{jk} (x_j)$.
  It follows that
  $e_i (x_i)_m = p_{mk} (e_{ik} (x_i)) \leq p_{mk} (e_{jk} (x_j)) =
  e_j (x_j)_m$, this time using the fact that $p_{mk}$ is continuous
  hence monotonic.

  For every $i \in I$, picking $k$ above $i$ and $m$, we have
  $e_i (x_i)_m = p_{mk} (e_{ik} (x_i))$, and the same argument as
  above specialized to $j \eqdef k$ yields
  $e_i (x_i)_m \leq p_{mk} (e_{kk} (x_k)) = p_{mk} (x_k) = x_m$.
  Hence $x_m$ is an upper bound of the directed family
  ${(e_i (x_i))_m}_{i \in I}$.  It remains to show that it is the
  least upper bound.  That is easy, since the least upper bound is in
  fact attained: for $i \eqdef m$,
  $e_i (x_i)_m = e_m (x_m)_m = p_{mk} (e_{mk} (x_m)) = x_m$.

  5. If $i \sqsubseteq j$, we claim that
  $(e_i \circ p_i)^{-1} (U) \subseteq (e_j \circ p_j)^{-1} (U)$.  For
  every $\vec x \in (e_i \circ p_i)^{-1} (U)$,
  $(e_i \circ p_i) (\vec x)$ is in $U$ and below
  $(e_j \circ p_j) (\vec x)$ by item~4.  Since open sets are
  upwards-closed, $\vec x$ is also in $(e_j \circ p_j)^{-1} (U)$.

  We profit from Lemma~\ref{lemma:prohorov:top}, item~3, and we write
  $U$ as $\dcup_{i \in I} p_i^{-1} (U_i)$, where each
  $U_i \eqdef p_i^* (U)$ is open in $X_i$.  For every $\vec x \in U$,
  there is an $i \in I$ such that $p_i (\vec x) \in U_i$.  Since
  $p_i \circ e_i = \identity {X_i}$,
  $p_i (\vec x) = p_i (e_i (p_i (\vec x)))$, and this shows that
  $e_i (p_i (\vec x))$ is in $p_i^{-1} (U_i)$, hence in $U$.
  Therefore $\vec x$ is in $(e_i \circ p_i)^{-1} (U)$.  We conclude
  that $U \subseteq \dcup_{i \in I} (e_i \circ p_i)^{-1} (U)$.

  In the converse direction, for every $i \in I$ and every $\vec x \in
  (e_i \circ p_i)^{-1} (U)$, $(e_i \circ p_i) (\vec x)$ is below $\vec
  x$ (by item~4) and in $U$, so $\vec x$ is in $U$.
\end{proof}

\begin{theorem}[Projective limits, ep-system case]
  \label{thm:prohorov:ep}
  Given an ep-system in $\Topcat$ consisting of spaces $X_i$,
  $i \in I$, and ep-pairs $(p_{ij}, e_{ij})$, $i \sqsubseteq j$ in
  $I$, let $X, {(p_i)}_{i \in I}$ be a projective limit of that
  ep-system.

  Assume that, for every $i \in I$, there is a continuous valuation
  $\nu_i$ on $X_i$, and that for all $i \sqsubseteq j$ in $I$, $\nu_i
  = p_{ij} [\nu_j]$.  Then there is a unique continuous valuation
  $\nu$ on $X$ such that, for every $i \in I$, $\nu_i = p_i [\nu]$.
\end{theorem}
\begin{proof}
  Up to isomorphism, we may assume that $X$ is the canonical
  projective limit
  $\{\vec x \in \prod_{i \in I} X_i \mid \forall i \sqsubseteq j, x_i
  = p_{ij} (x_j)\}$.  For every open subset $U$ of $X$, we use
  Lemma~\ref{lemma:epsystem}, item~5, and this leads us to define
  $\nu (U)$ as $\dsup_{i \in I} \nu_i (e_i^{-1} (U))$.

  We check that the family ${(\nu_i (e_i^{-1} (U)))}_{i \in I}$ is
  directed.  In order to do so, we show that for all $i \sqsubseteq j$
  in $I$, $\nu_i (e_i^{-1} (U)) \leq \nu_j (e_j^{-1} (U))$.  For every
  $x \in p_{ij}^{-1} (e_i^{-1} (U))$, $e_i (p_{ij} (x))$ is in $U$.
  It is equal to $e_j (e_{ij} (p_{ij} (x)))$ by
  Lemma~\ref{lemma:epsystem}, item~3, hence is below $e_j (x)$ by
  item~2 of the same lemma and the fact that $e_j$, being continuous,
  is monotonic.  Since $U$ is upwards-closed, $e_j (x)$ is in $U$.
  This shows that $x$ is in $e_j^{-1} (U)$.  Therefore
  $p_{ij}^{-1} (e_i^{-1} (U)) \subseteq e_j^{-1} (U)$, and hence
  $\nu_i (e_i^{-1} (U)) = p_{ij} [\nu_j] (e_i^{-1} (U)) = \nu_j
  (p_{ij}^{-1} (e_i^{-1} (U))) \leq \nu_j (e_j^{-1} (U))$.

  Let us verify that $\nu$ is Scott-continuous from $\Open X$ to
  $\creal$.  It is clear that $\nu$ is monotonic.  Let
  ${(U_j)}_{j \in J}$ be a directed family of open subsets of $X$, and
  let $U$ be its union.  Then
  $e_i^{-1} (U) = \dcup_{j \in J} e_i^{-1} (U_j)$ for every $i \in I$,
  and since $\nu_i$ is itself Scott-continuous,
  $\nu (U) = \dsup_{i \in I} \dsup_{j \in J} \nu_i (e_i^{-1} (U_j))$.
  This is equal to
  $\dsup_{j \in J} \dsup_{i \in I} \nu_i (e_i^{-1} (U_j)) = \dsup_{j
    \in J} \nu (U_j)$.

  By the second part of Proposition~\ref{prop:prohorov:unique}, $\nu$
  is a continuous valuation.  It is unique by the first part of the
  same proposition.
\end{proof}

\section{A Daniell-Kolmogorov theorem}
\label{sec:dani-kolm-theor}

Let $X_i$, $i \in I$, be a family of topological spaces (resp.,
measurable spaces).  For every finite subset $J$ of $I$, one can form
the finite product $X_J \eqdef \prod_{i \in J} X_i$, and define bonding maps
$p_{JJ'} \colon X_{J'} \to X_J$ for all $J \subseteq J'$ in $\Pfin
(I)$ by $p_{JJ'} ({(x_i)}_{i \in J'}) \eqdef {(x_i)}_{i \in J}$.  Then
$\prod_{i \in I} X_i$ is a projective limit of the projective system
we have just defined, together with the maps $p_J \colon \vec x \in
\prod_{i \in X_i} \to {(x_i)}_{i \in J}$, $J \in \Pfin (I)$.

We first deal with a very particular case, where every $X_i$ is
\emph{pointed}, namely, has a least element with respect to $\leq$.
\begin{proposition}
  \label{prop:kolmogorov:pointed}
  Let ${(X_i)}_{i \in I}$ be a family of pointed topological spaces,
  $X_J \eqdef \prod_{i \in J} X_i$ for every $J \in \Pfin (I)$,
  $p_{JJ'} \colon {(x_i)}_{i \in J'} \in X_{J'} \to {(x_i)}_{i \in J}
  \in X_J$ for all $J \subseteq J'$ in $\Pfin (I)$,
  $X \eqdef \prod_{i \in I} X_i$, and
  $p_J \colon \vec x \in X \to {(x_i)}_{i \in J} \in X_J$ for every
  $J \in \Pfin (I)$.

  For every family of continuous valuations $\nu_J$ on $X_J$,
  $J \in \Pfin (I)$, such that $\nu_J = p_{JJ'} [\nu_{J'}]$ for all
  $J \subseteq J'$ in $\Pfin (I)$, there is a unique continuous
  valuation $\nu$ on $X$ such that $p_J [\nu] = \nu_J$ for every
  $J \in \Pfin (I)$.
\end{proposition}
\begin{proof}
  We write $\bot$ for least elements of $X_i$, whatever $i \in I$ is
  chosen.
  For all $J \subseteq J'$ in $\Pfin (I)$, define
  $e_{JJ'} \colon X_J \to X_{J'}$ so that the $j$th component of
  $e_{JJ'} (\vec x)$ is $x_j$ if $j \in J$, and $\bot$ otherwise.
  Each pair $(p_{JJ'}, e_{JJ'})$ is an ep-pair, and when
  $J \subseteq J'$ varies, they form an ep-system.  We conclude by
  Theorem~\ref{thm:prohorov:ep}.
%
%
\end{proof}

In order to deal with the general case, we require the notion of
support of a continuous valuation $\nu$ on a space $X$.
\begin{proposition}
  \label{prop:supp}
  Let $X$ be a topological space, $A$ be a subset of $X$,
  $\eta \colon A \to X$ be the inclusion map, and let $\nu$ be a
  valuation on $X$.  The following are equivalent:
  \begin{enumerate}
  \item for all open $U, V \in \Open X$ such that
    $U \cap A = V \cap A$, $\nu (U) = \nu (V)$;
  \item $\nu$ is \emph{supported on $A$} (a.k.a., $A$ is a
    \emph{support} of $\nu$), namely, there is a valuation $\mu$ on
    $A$ such that $\nu = \eta [\mu]$.
  \end{enumerate}
  In that case, the valuation $\mu$ in item~2 is unique,
  and characterized by the formula $\mu (U \cap A) = \nu (U)$ for
  every open subset $U$ of $X$.

  Furthermore, if $\nu$ is continuous, then $\mu$ is continuous.
\end{proposition}
\begin{proof}
  2 implies 1: for all $U, V \in \Open X$ such that
  $U \cap A = V \cap A$,
  $\nu (U) = \mu (U \cap A) = \mu (V \cap A) = \nu (V)$.

  1 implies 2: define a function $\mu$ from $\Open A$ to $\creal$ by
  $\mu (U \cap A) \eqdef \nu (U)$.  If an element of $\Open A$ can be
  written both as $U \cap A$ and as $V \cap A$, for two open sets
  $U, V \in \Open X$, then $\nu (U) = \nu (V)$, so the values
  $\mu (U \cap A)$ and $\mu (V \cap A)$ are the same, by item~1,
  showing that our definition is not ambiguous.

  Clearly, $\mu$ is strict.  For monotonicity, assume $U \cap A
  \subseteq V \cap A$.  It may not be the case that $U$ is included in
  $V$.  However, replacing $U$ by $U \cap V$ yields an element of
  $\Open X$ that has the same intersection with $A$, so $\mu (U \cap A) =
  \mu ((U \cap V) \cap A) = \nu (U \cap V) \leq \nu (V) = \mu (V \cap
  A)$.

  As far as modularity is concerned,
  $\mu ((U \cap A) \cup (V \cap A)) + \mu ((U \cap A) \cap (V \cap A))
  = \mu ((U \cup V) \cap A) + \mu ((U \cap V) \cap A) = \nu (U \cup V)
  + \nu (U \cap V) = \nu (U) + \nu (V) = \mu (U \cap A) + \mu (V \cap
  A)$.

  Let us assume that $\nu$ is Scott-continuous.  We prove that $\mu$
  is Scott-continuous as follows.  Let ${(U_i \cap A)}_{i \in I}$ be a
  directed family of open subsets of $A$.  As for monotonicity, the
  family ${(U_i)}_{i \in I}$ of open subsets of $X$ may fail to be
  directed.  For $i, j \in I$, let $j \preceq i$ if and only if
  $U_j \cap A \subseteq U_i \cap A$, and replace each $U_i$ by the
  open subset $V_i = \bigcup_{j \preceq i} U_j$.  Since $j \preceq i$
  implies $V_j \subseteq V_i$, the family ${(V_i)}_{i \in I}$ is
  directed.  Moreover,
  $V_i \cap A = \bigcup_{j \preceq i} (U_j \cap A) \subseteq U_i \cap
  A$ (by the definition of $\preceq$) $\subseteq V_i \cap A$ (by the
  definition of $V_i$), so $V_i \cap A = U_i \cap A$.  We can now
  conclude:
  $\mu (\bigcup_{i \in I} (U_i \cap A)) = \mu (\bigcup_{i \in I} (V_i
  \cap A)) = \nu (\dcup_{i \in I} V_i) = \dsup_{i \in I} \nu (V_i) =
  \dsup_{i \in I} \mu (V_i \cap A)$.
\end{proof}

The following proof will use supports in a crucial way.  It will also
rely on a construction of spaces that are never Hausdorff: for a
topological space $X$, let its \emph{lift} $X_\bot$ be $X$ plus an
extra point $\bot$; the open subsets of $X_\bot$ are those of $X$,
plus $X_\bot$.  The specialization preordering of $X_\bot$ is given by
$x \leq y$ (in $X_\bot$) if and only if $x=\bot$, or $x$ and $y$ are
in $X$ and $x \leq y$ in $X$.

\begin{theorem}[\`a la Daniell-Kolmogorov]
  \label{thm:kolmogorov}
  Let ${(X_i)}_{i \in I}$ be a family of topological spaces,
  $X_J \eqdef \prod_{i \in J} X_i$ for every $J \in \Pfin (I)$,
  $p_{JJ'} \colon {(x_i)}_{i \in J'} \in X_{J'} \to {(x_i)}_{i \in J}
  \in X_J$ for all $J \subseteq J'$ in $\Pfin (I)$,
  $X \eqdef \prod_{i \in I} X_i$, and
  $p_J \colon \vec x \in X \to {(x_i)}_{i \in J} \in X_J$ for every
  $J \in \Pfin (I)$.

  For every family of continuous valuations $\nu_J$ on $X_J$,
  $J \in \Pfin (I)$, such that $\nu_J = p_{JJ'} [\nu_{J'}]$ for all
  $J \subseteq J'$ in $\Pfin (I)$, there is a unique continuous
  valuation $\nu$ on $X$ such that $p_J [\nu] = \nu_J$ for every
  $J \in \Pfin (I)$.
 \end{theorem}
\begin{proof}
  It is unique by Proposition~\ref{prop:prohorov:unique}.  Consider
  the lift $X_{i\bot}$ of $X_i$, and form the topological space
  $Y \eqdef \prod_{i \in I} X_{i\bot}$.

  For each $J \in \Pfin (I)$, let also $Y_J$ be
  $\prod_{i \in J} X_{i\bot}$.  The inclusion map from $X_i$ into
  $X_{i\bot}$ is continuous, hence so is the inclusion map $\eta_J$
  from $X_J$ into $Y_J$.  We form the image $\eta_J [\nu_J]$, so
  $\eta_J [\nu_J] (U) \eqdef \nu_J (U \cap X_J)$ for every open subset
  $U$ of $Y_J$.  Let us write $q_{JJ'}$ for the map that sends every
  ${(y_i)}_{i \in J'}$ in $Y_{J'}$ to ${(y_i)}_{i \in J}$ in $Y_J$ for
  all $J \subseteq J'$ in $\Pfin (I)$.  Note that
  $q_{JJ'} \circ \eta_{J'} = \eta_J \circ p_{JJ'}$.  Therefore, for
  all $J \subseteq J'$ in $\Pfin (I)$,
  $q_{JJ'} [\eta_{J'} [\nu_{J'}]] =\eta_J [p_{JJ'} [\nu_{J'}]] =
  \eta_J [\nu_J]$.  Thus we have a projective system
  ${(q_{JJ'} \colon (Y_{J'}, \eta_{J'} [\nu_{J'}]) \to (Y_J, \eta_J
    [\nu_J]))}_{J \subseteq J' \in \Pfin (I)}$ of pointed valued
  spaces.  Let $q_J$ map every $\vec y \in Y$ to
  ${(y_i)}_{i \in J} \in Y_J$, for every $J \in \Pfin (I)$.  By
  Proposition~\ref{prop:kolmogorov:pointed}, there is a probability
  valuation $\tilde\nu$ on $Y$ such that, for every $J \in \Pfin (I)$,
  $q_J [\tilde\nu] = \eta_J [\nu_J]$.

  We claim that $\tilde\nu$ is supported on
  $X \eqdef \prod_{i \in I} X_i$.  To that end, we first observe that
  for every open subset $U$ of $Y$, there is a largest open subset $V$
  of $Y$ such that $U \cap X = V \cap X$.  This is just the union of
  all the open subsets $V$ such that $U \cap X = V \cap X$, but we can
  characterize it as follows.

  For every basic open subset $B \eqdef \prod_{i \in I} U_i$ of $Y$,
  let $\tilde B \eqdef \prod_{i \in I} \tilde U_i$, where
  $\tilde U_i \eqdef X_{i\bot}$ if $U_i=X_i$ or $U_i=X_{i\bot}$, and
  $\tilde U_i \eqdef U_i$ otherwise.  Note that
  $B \cap X = \tilde B \cap X$.  For every open subset $U$ of $Y$,
  define $\tilde U$ as the union of all $\tilde B$ where $B$ ranges
  over the basic open subsets included in $U$.  In particular,
  $U \cap X = \tilde U \cap X$.  We claim that $\tilde U$ is the
  largest open subset $V$ such that $U \cap X = V \cap X$.  Indeed,
  consider any open subset $V$ such that $U \cap X = V \cap X$.  For
  every basic open subset $B \eqdef \prod_{i \in I} U_i$ included in
  $V$, $B \cap X = \prod_{i \in I} (U_i \cap X_i)$ is included in
  $V \cap X = U \cap X$, hence in $U$.  It follows that
  $\widetilde {(B \cap X)}$ is included in $\tilde U$.  But
  $\widetilde {(B \cap X)} = \prod_{i \in I} \widetilde {(U_i \cap
    X_i)}$ contains $B$, since for each $i$, $U_i$ is included in
  $\widetilde {(U_i \cap X_i)}$, as one checks easily.  Hence $B$ is
  included in $\tilde U$.  Since $B$ is arbitrary, $V$ is included in
  $\tilde U$.

  It follows that for any two open subsets $U$ and $V$ of $Y$, if
  $U \cap X = V \cap X$ then $\tilde U = \tilde V$.  Explicitly, among
  all the open subsets whose intersection with $X$ are equal to
  $U \cap X = V \cap X$, $\tilde U$ and $\tilde V$ both are the
  largest.

  Hence, in order to show that $U \cap X = V \cap X$ implies
  $\tilde\nu (U) = \tilde\nu (V)$, it suffices to show that
  $\tilde\nu (U) =\tilde\nu (\tilde U)$ (and symmetrically,
  $\tilde\nu (V) = \tilde\nu (\tilde V)$).  We do this in two steps.

  Step 1.  For every finite union $U \eqdef \bigcup_{k=1}^n B_k$ where
  each $B_k$ is a basic open subset of $Y$, let us write $B_k$ as
  $\prod_{i \in I} U_{ki}$, where each $U_{ki}$ is open in $X_{i\bot}$
  and $U_{ki} = X_{i\bot}$ for every $i \in I$ outside some finite set
  $J_k$.  Let $J \eqdef J_1 \cup \cdots \cup J_n$, so that
  $U_{ki} = X_{i\bot}$ for every $i \in I \diff J$ and every $k$,
  $1\leq k\leq n$.  Then $U = q_J^{-1} (V)$ where
  $V \eqdef \bigcup_{k=1}^n \prod_{i \in J} U_{ki}$ (note that $i$ now
  ranges over $J$ instead of $I$), and $\bigcup_{k=1}^n \tilde B_k = q_J^{-1} (W)$ where
  $W \eqdef \bigcup_{k=1}^n \prod_{i \in J} \tilde U_{ki}$.  Since
  $U_{ki}$ and $\tilde U_{ki}$ contain the same points from $X_i$,
  $\eta_J^{-1} (\prod_{i \in J} \tilde U_{ki}) = \eta_J^{-1} (\prod_{i
    \in J} U_{ki})$, for every $k$, so
  $\eta_J^{-1} (V) = \eta_J^{-1} (W)$.

  Recall that $q_J [\tilde\nu] = \eta_J [\nu_J]$.  Therefore
  $\tilde\nu (U) = \tilde\nu (q_J^{-1} (V)) = q_J [\tilde\nu] (V) =
  \eta_J [\nu_J] (V) = \nu_J (\eta_J^{-1} (V))$, and similarly
  $\tilde\nu (\bigcup_{k=1}^n \tilde B_k) = \nu_J (\eta_J^{-1} (W))$.
  Since $\eta_J^{-1} (V) = \eta_J^{-1} (W)$,
  $\tilde\nu (U) = \tilde\nu (\bigcup_{k=1}^n \tilde B_k)$.

  Step 2.  For general open subsets $U$ of $Y$, recall that $\tilde U$
  is the union of all sets $\tilde B$ where $B$ ranges over the basic
  open subsets included in $U$.  Organize the finite unions of such
  sets $B$ as a directed family ${(U_j)}_{j \in J}$.  Then
  $U = \dcup_{j \in J} U_j$.  Also, for each $j \in J$, if $U_j$ is a
  finite union $\bigcup_{k=1}^n B_k$, where each $B_k$ is a basic open
  subset, then let $V_j \eqdef \bigcup_{k=1}^n \tilde B_k$.  By
  construction, $\tilde U = \dcup_{i \in I} V_i$.  By step~2,
  $\tilde\nu (V_j) = \tilde\nu (U_j)$.  Since $\tilde\nu$ is
  continuous,
  $\tilde\nu (\tilde U) = \dsup_{j \in J} \tilde\nu (V_j) = \dsup_{j
    \in J} \tilde\nu (U_j) = \tilde\nu (U)$.

  This finishes to show that $\tilde\nu$ is supported on $X$.  Hence
  there is a continuous valuation $\nu$ on $X$, with the subspace
  topology from $Y$, such that $\nu (U \cap X) = \tilde\nu (U)$ for
  every open subset $U$ of $Y$, by Proposition~\ref{prop:supp}.  It is
  easy to see that the subspace topology on $X$ coincides with the
  product topology.  Let $\eta$ be the inclusion map from $X$ into
  $Y$.  Then $\eta [\nu] = \tilde\nu$.  For every $J \in \Pfin (I)$,
  $q_J [\tilde\nu] = \eta_J [\nu_J]$, so $\eta_J [\nu_J]$ is equal to
  $q_J [\eta [\nu]] = \eta_J [p_J [\nu]]$.  We now note that $X_i$ is
  open in $X_{i\bot}$ for every $i$, so $X_J$ is an open rectangle in
  $Y_J$.  It follows that every open subset $U$ of $X_J$ is also open
  in $Y_J$, so $\eta_J^{-1} (U) = U$.  Therefore
  $\nu_J (U) = \nu_J (\eta_J^{-1} (U)) = \eta_J [\nu_J] (U) = \eta_J
  [p_J [\nu]] = p_J [\nu] (\eta_J^{-1} (U)) = p_J [\nu] (U)$.  Since
  $U$ is arbitrary, $p_J [\nu] = \nu_J$.
\end{proof}

\section{Tightness and uniform tightness}
\label{sec:tightn-unif-tightn}

We now attack the general form of Problem~\ref{pb:prohorov}, looking
only for tight valuations, a notion we now define.

For every compact saturated subset $Q$ of a space $X$, we write
$\blacksquare Q$ for the family of open neighborhoods of $Q$.  We also
write $\Smyth X$ for the set of all compact saturated subsets of $X$.

For every map $\nu \colon \Open X \to \creal$, let $\nu^\bullet$ map
every $Q \in \Smyth X$ to $\finf_{U \in \blacksquare Q} \nu (U)$.
That construction appears as the $\nu^\dagger$ construction in
\cite{Tix:bewertung}, and as the $\nu^*$ construction in
\cite{KL:measureext}.  This is a strict, monotonic map that is not in
general modular or Scott-continuous.

Dually, we write $\Box U$ for the family of all compact saturated
subsets of the open set $U$, and we define $\mu^\circ (U)$ as
$\dsup_{Q \in \Box U} \mu (Q)$, for any function
$\mu \colon \Smyth X \to \creal$.

A bounded measure $\mu$ on a Hausdorff topological space $X$ is
\emph{tight} if and only if for every $\epsilon > 0$ there is a
compact set $K$ such that $\mu (K) > \mu (X) - \epsilon$.  In
non-Hausdorff spaces, compact subsets need not be measurable, and we
use $\nu^\bullet$ to measure compact subsets instead of $\nu$.  Hence
we define the following.  Here $\ll$ is the way-below relation on
$\creal$; we have $r \ll s$ if and only if $r=0$ or $r < s$.
\begin{definition}
  \label{defn:inner:regular}
  Let $X$ be a topological space.  A map $\nu$ from $\Open X$ to
  $\creal$ is \emph{tight} if and only if, for every open subset $U$
  of $X$ and every $r \ll \nu (U)$, there is a compact saturated
  subset $Q$ of $X$ such that $Q \subseteq U$ and $r \leq \nu^\bullet (Q)$.
\end{definition}

Given a family ${(Q_i)}_{i \in I}$ of compact saturated subsets of a
space $X$, $\bigcup_{i \in I} \blacksquare Q_i$ is always a Scott-open
family in $\Open X$.  A space is \emph{consonant} if and only if every
Scott-open family in $\Open X$ is of this form.  Equivalently, if and
only if for every Scott-open family $\mathcal F$ of open subsets of
$X$, for every $U \in \mathcal F$, there is a compact saturated subset
$Q$ of $X$ such that $Q \subseteq U$ and
$\blacksquare Q \subseteq \mathcal F$.

Every locally compact space is consonant, and the
\emph{Dolecki-Greco-Lechicki theorem} states that all regular \v
Cech-complete spaces are consonant \cite{DGL:consonant}.  The latter
include all complete metric spaces, even not locally compact, in their
open ball topology.

The following lemma says that tight maps are the same thing as
Scott-continuous maps, provided $X$ is consonant.
\begin{lemma}
  \label{lemma:inner:regular}
  Let $X$ be a topological space, and $\nu$ be a map from $\Open X$ to
  $\creal$.
  \begin{enumerate}
  \item If $\nu$ is tight, then $\nu$ is Scott-continuous.
  \item $\nu$ is tight if and only if it is of the form $\mu^\circ$
    for some map $\mu \colon \Smyth X \to \creal$;
  \item $\nu$ is tight if and only if $\nu = \nu^{\bullet\circ}$---if
    and only if $\nu \leq \nu^{\bullet\circ}$, the converse inequality
    being true for every map $\nu$.
  \item If $X$ is consonant, and $\nu$ is Scott-continuous, then $\nu$
    is tight.
  \end{enumerate}
\end{lemma}
\begin{proof}
  1. We first show that $\nu$ is monotonic.  Let $U \subseteq V$ be
  open subsets of $X$.  For every $r \ll \nu (U)$, there is a compact
  saturated set $Q \subseteq U$ such that
  $r \leq \nu^\bullet (Q) \leq \nu (V)$.  Taking suprema over $r$, we
  obtain $\nu (U) \leq \nu (V)$.

  Next, let ${(U_i)}_{i \in I}$ be a directed family of open subsets
  of $X$, and let $U$ be its union.  The inequality $\dsup_{i \in I}
  \nu (U_i) \leq \nu (U)$ is immediate from the fact that $\nu$ is
  monotonic.  Conversely, for every $r \ll \nu (U)$, we find a compact
  saturated set $Q \subseteq U$ such that $r \leq \nu (V)$ for every
  open neighborhood $V$ of $Q$.  Since $Q$ is compact, $Q \subseteq
  U_i$ for some $i \in I$, so $r \leq \nu (U_i)$.  It follows that $r
  \leq \dsup_{i \in I} \nu (U_i)$, and, by taking suprema over $r$,
  that $\nu (U) \leq \dsup_{i \in I} \nu (U_i)$.

  2 and 3.  We first show that any map of the form $\mu^\circ$ is
  tight.  For every open subset $U$ of $X$, and every
  $r \ll \mu^\circ (U)$, by definition there is a $Q \in \Box U$ such
  that $r \leq \mu (Q)$.  For every open neighborhood $V$ of $Q$,
  $\mu (Q) \leq \mu^\circ (V)$, so $r \leq \mu^\circ (V)$.  That shows
  the if direction of item~2.

  We proceed with item~3.  In one direction, we show that
  $\nu^{\bullet\circ} \leq \nu$, for every $\nu$.  Let $U$ be open in
  $X$.  For every $Q \subseteq U$, $\nu^\bullet (Q) \leq \nu (U)$, by
  definition of $\nu^\bullet$.  Then
  $\nu^{\bullet\circ} (U) = \dsup_{Q \in \Box U} \nu^\bullet (Q) \leq
  \dsup_{Q \in \Box U} \nu (U) = \nu (U)$.  (The last equality holds
  because the supremum is non-empty.  It is non-empty because we can
  take the empty set for $Q$.)

  Rewriting the definition slightly, $\nu$ is tight if and only if,
  for every $U \in \Open X$, for every $r \ll \nu (U)$, there is a
  $Q \in \Box U$ such that $r \leq \nu^\bullet (Q)$.

  If that is so, then for every $r \ll \nu (U)$, $r$ is less than or
  equal to
  $\sup_{Q \in \Box U} \nu^\bullet (Q) = \nu^{\bullet\circ} (U)$.  By
  taking suprema over $r$, $\nu (U) \leq \nu^{\bullet\circ} (U)$, and
  we have seen that the converse inequality always holds.

  Conversely, if $\nu = \nu^{\bullet\circ}$, i.e., if $\nu =
  \mu^\circ$ where $\mu \eqdef \nu^\bullet$, then we have seen that
  $\nu$ is tight.  This shows item~3.  We complete the proof of item~2
  by noting that if $\nu$ is tight, by item~3 $\nu$ is of the form
  $\mu^\circ$ where $\mu \eqdef \nu^\bullet$.

  4. Let $X$ be consonant and $\nu$ be Scott-continuous. Let also
  $U \in \Open X$ and $r \ll \nu (U)$.  The family
  $\mathcal F \eqdef \nu^{-1} (\uuarrow r) = \{V \in \Open X \mid r
  \ll \nu (V)\}$ contains $U$ and is Scott-open.  Since $X$ is
  consonant, there is a $Q \in \Box U$ such that
  $\blacksquare Q \subseteq \mathcal F$.  The latter inclusion means
  that for every open subset $V$ such that $Q \subseteq V$, $r \ll \nu
  (V)$.  In particular, $r \leq \nu (V)$, and since $V$ is arbitrary,
  $r \leq \nu^\bullet (Q)$.
\end{proof}

The next lemma gives a necessary condition for
Problem~\ref{pb:prohorov} to have a tight solution.  In the proof, we
use the following easily proved fact.
\begin{fact}
  \label{fact:pQ:U}
  For every continuous map $p \colon X \to Y$ between topological
  spaces, for every compact saturated subset $Q$ of $X$, $\upc p [Q]$
  is compact saturated in $Y$.  For every open subset $U$ of $Y$,
  $\upc p [Q] \subseteq U$ if and only if $Q \subseteq p^{-1} (U)$.  \qed
\end{fact}

\begin{lemma}
  \label{lemma:prohorov:onlyif}
  Let ${(p_{ij} \colon X_j \to X_i)}_{i \sqsubseteq j \in I}$ be a
  projective system of topological spaces, and $X, {(p_i)}_{i \in I}$
  be a projective limit of that system.  Let
  $\nu \colon \Open X \to \creal$ be a tight map, and
  $\nu_i \eqdef p_i [\nu]$ for every $i \in I$.

  Then the family ${(\nu_i)}_{i \in I}$ satisfies the following
  condition of \emph{uniform tightness}: for every $i \in I$, for
  every open subset $U$ of $X_i$ and every $r \ll \nu_i (U)$, there is
  a compact saturated subset $Q$ of $X$ such that
  $\upc p_i [Q] \subseteq U$ and such that for every $j \in I$,
  $r \leq \nu_j^\bullet (\upc p_j [Q])$.

  In particular, every $\nu_i$ is tight.
\end{lemma}
\begin{proof}
  Since $\nu$ is tight, for every $i \in I$, for every open subset $U$
  of $X_i$, for every $r \ll \nu_i (U) = \nu (p_i^{-1} (U))$, there is
  a compact saturated subset $Q$ of $X$ such that
  $Q \subseteq p_i^{-1} (U)$ and $r \leq \nu^\bullet (Q)$.  By
  Fact~\ref{fact:pQ:U}, $\upc p_i [Q]$ is compact, saturated, and
  included in $U$.  Since $r \leq \nu^\bullet (Q)$, for every open
  subset $V$ of $X$ that contains $Q$, we must have $r \leq \nu (V)$.
  In particular, for every $j \in I$, for every open subset $W$ of
  $X_j$, if $\upc p_j [Q] \subseteq W$, then
  $Q \subseteq p_j^{-1} (W)$ by Fact~\ref{fact:pQ:U}, so
  $r \leq \nu (p_j^{-1} (W)) = \nu_j (W)$.  Taking infima over all
  open neighborhoods $W$ of $\upc p_j [Q]$,
  $r \leq \nu_j^\bullet (\upc p_j [Q])$.

  Taking $j \eqdef i$, we obtain that if $r \ll \nu_i (U)$ then there
  is a compact saturated subset $Q$ of $X$ such that
  $r \leq \nu_i^\bullet (\upc p_i [Q])$.  Hence there is a compact
  saturated subset $Q_i \eqdef \upc p_i [Q]$ of $X_i$ such that
  $r \leq \nu_i^\bullet (Q_i)$, showing that $\nu_i$ is tight.
\end{proof}

The uniform tightness condition is somewhat impenetrable.  The
following gives both a more synthetic condition.
\begin{lemma}
  \label{lemma:unif:tight}
  Let ${(p_{ij} \colon X_j \to X_i)}_{i \sqsubseteq j \in I}$ be a
  projective system of topological spaces, and $X, {(p_i)}_{i \in I}$
  be a projective limit of that system.  Let also $\nu_i$ be
  Scott-continuous maps from $\Open {X_i}$ to $\creal$, for each
  $i \in I$, and assume that for all $i \sqsubseteq j$ in $I$,
  $\nu_i = p_{ij} [\nu_j]$.

  Define $\mu \colon \Smyth X \to \creal$ by
  $\mu (Q) \eqdef \finf_{i \in I} \nu_i^\bullet (\upc p_i [Q])$.  Then
  $p_i [\mu^\circ] \leq \nu_i$ for every $i \in I$, and the following
  are equivalent:
  \begin{enumerate}
  \item ${(\nu_i)}_{i \in I}$ is uniformly tight;
  \item for every $i \in I$, $\nu_i \leq p_i [\mu^\circ]$;
  \item for every $i \in I$, $\nu_i = p_i [\mu^\circ]$.
  \end{enumerate}
\end{lemma}
\begin{proof}
  In order to assess that the definition of $\mu$ makes sense, we
  first check that ${(\nu_i^\bullet (\upc p_i [Q]))}_{i \in I}$ is a
  filtered family.  For all $i \sqsubseteq j$ in $I$, every open
  subset $U$ that contains $\upc p_i [Q]$ is such that $p_i^{-1} (U)$
  contains $Q$ by Fact~\ref{fact:pQ:U}, so
  $p_j^{-1} (p_{ij}^{-1} (U))$ contains $Q$.  It follows that
  $p_{ij}^{-1} (U)$ contains $\upc p_j [Q]$.  Therefore
  $\nu_i (U) = p_{ij} [\nu_j] (U) = \nu_j (p_{ij}^{-1} (U)) \geq
  \nu_j^\bullet (\upc p_j [Q])$.  Taking infima over all open
  neighborhoods $U$ of $\upc p_i [Q]$, we obtain that
  $\nu_i^\bullet (\upc p_i [Q]) \geq \nu_j^\bullet (\upc p_j [Q])$.

  Let us show that $p_i [\mu^\circ] \leq \nu_i$.  For every open
  subset $U$ of $X_i$,
  $p_i [\mu^\circ] (U) = \mu^\circ (p_i^{-1} (U)) = \sup_{Q \in \Box
    {p_i^{-1} (U)}} \mu (Q)$.  For every $Q \in \Box {p_i^{-1} (U)}$,
  $\upc p_i [Q]$ is included in $U$ by Fact~\ref{fact:pQ:U}, so by
  definition of $\mu$,
  $\mu (Q) \leq \nu_i^\bullet (\upc p_i [Q]) \leq \nu_i (U)$.

  $1 \limp 2$.  Let $U$ be any open subset of $X_i$, and let
  $r \ll \nu_i (U)$ be arbitrary.  By uniform tightness, there is a
  compact saturated subset $Q$ of $X$ such that
  $\upc p_i [Q] \subseteq U$ and such that for every $j \in I$,
  $r \leq \nu_j^\bullet (\upc p_j [Q])$.  Equivalently,
  $Q \subseteq p_i^{-1} (U)$ and
  $r \leq \inf_{j \in I} \nu_j^\bullet (\upc p_j [Q]) = \mu (Q)$.
  Therefore $r \leq \mu^\circ (p_i^{-1} (U)) = p_i [\mu^\circ] (U)$.  Taking
  suprema over $r$, $\nu_i (U) \leq p_i [\mu^\circ] (U)$.

  $2 \limp 3$ is obvious, considering that the converse inequality
  always holds.

  $3 \limp 1$.  By Lemma~\ref{lemma:inner:regular}, item~1,
  $\mu^\circ$ is tight.  This allows us to apply
  Lemma~\ref{lemma:prohorov:onlyif} with $\nu \eqdef \mu^\circ$, and
  to conclude that ${(\nu_i)}_{i \in I}$ is uniformly tight.
\end{proof}

\begin{remark}
  \label{rem:unif:tight:kripke}
  A family ${(\nu_i)}_{i \in I}$ is uniformly tight if and only if for
  every $i \in I$, for every open subset $U$ of $X_i$ and every
  $r \ll \nu_i (U)$, there is a compact saturated subset $Q$ of $X$
  such that $\upc p_i [Q] \subseteq U$ and such that for every
  $j \in \upc i$, $r \leq \nu_j^\bullet (\upc p_j [Q])$.  In other
  words, we only need to do the final check for indices $j$ above $i$,
  not for all indices $j \in I$.  Indeed, assume that
  $r \leq \nu_j^\bullet (\upc p_j [Q])$ for every $j \in \upc i$.
  Pick any $j \in I$.  By directedness, there is a $k \in I$ above $i$
  and $j$.  Then
  $r \leq \nu_k^\bullet (\upc p_k [Q]) \leq \nu_j^\bullet (\upc p_j
  [Q])$.
\end{remark}

We can now give a solution to Problem~\ref{pb:prohorov} under a
uniform tightness assumption, \`a la Prohorov \cite{Prohorov:projlim}.
\begin{theorem}[\`a la Prohorov]
  \label{thm:prohorov}
  Let
  ${(p_{ij} \colon (X_j, \nu_j) \to (X_i, \nu_i))}_{i \sqsubseteq j
    \in I}$ be a projective system of valued spaces.  Let
  $X, {(p_i)}_{i \in I}$ be a projective limit of the underlying
  projective system
  ${(p_{ij} \colon X_j \to X_i)}_{i \sqsubseteq j \in I}$.

  If the family ${(\nu_i)}_{i \in I}$ is uniformly tight, then there
  is a unique continuous valuation $\nu$ on $X$ such that for every
  $i \in I$, $\nu_i = p_i [\nu]$.  Moreover, $\nu$ is tight.
\end{theorem}
Lemma~\ref{lemma:unif:tight} states that uniform tightness is a
necessary condition.

\begin{proof}
  Uniqueness is by Proposition~\ref{prop:prohorov:unique}.  Define
  $\mu$ as in Lemma~\ref{lemma:unif:tight}.  The implication
  $1 \limp 2$ there shows that $\nu \eqdef \mu^\circ$ is such that for
  every $i \in I$, $\nu_i = p_i [\nu]$.  By
  Lemma~\ref{lemma:inner:regular}, item~2, $\nu$ is tight, and by the
  second part of Proposition~\ref{prop:prohorov:unique}, $\nu$ is a
  continuous valuation.
\end{proof}

In brief, Theorem~\ref{thm:prohorov} states that
Problem~\ref{pb:prohorov} has a unique solution if
${(\nu_i)}_{i \in I}$ is uniformly tight.  We shall see two canonical
cases where every family of (tight) continuous valuations is uniformly
tight.

\section{Steenrod's theorem}
\label{sec:steendrods-theorem}

Uniform tightness requires one to build a compact saturated subset $Q$
in a projective limit.  Steenrod's theorem
\cite[Theorem~2.1]{Steenrod:univ:homol} will allow us to build $Q$,
as a projective limit of compact spaces.  We only need to assume
sobriety on top of compactness.

Steenrod does not assume sobriety, but his topological spaces are
$T_1$.  The proof of Theorem~2.1 of \cite{Steenrod:univ:homol} seems
to contain a gap, however, as it is claimed that certain images of
compact sets by continuous maps are closed (``Now each $A_1^\beta$ is
closed'', first line of proof of Theorem~2.1).  Repairing this, Stone
establishes theorems on projective limits of compact $T_1$ spaces and
\emph{closed} continuous bonding maps \cite{Stone:limproj:compact},
but those are not the kind that we are interested in.  Instead, we use
the following corrected version of Steenrod's theorem, due to Fujiwara
and Kato \cite[Theorem~2.2.20]{FK:rigid:geometry}.  (The authors state
that the theorem and its proof were suggested by O. Gabber.)

\begin{theorem}[Steenrod's theorem, \cite{FK:rigid:geometry}]
  \label{thm:steenrod}
  The canonical projective limit $Q, {(p_i)}_{i \in I}$ of any
  projective system
  ${(p_{ij} \colon Q_j \to Q_i)}_{i \sqsubseteq j \in I}$ of compact
  sober spaces is compact.  It is non-empty if every $Q_i$ is
  non-empty.  \qed
\end{theorem}

\begin{remark}
  \label{rem:steenrod:sober}
  Sobriety is essential in Theorem~\ref{thm:steenrod}, as the
  following example, due to Stone
  \cite[Example~3]{Stone:limproj:compact}, shows.  Let
  ${(p_{mn} \colon X_n \to X_m)}_{m \leq n \in \nat}$ be the following
  projective system.  Each $X_n$ is $\nat$, with a cofinite-like
  topology: its closed subsets $C$ are those such that $C \cap \upc n$
  is finite or equal to the whole of $\upc n$.  The bonding maps
  $p_{mn}$ are all identity maps.  A projective limit of that system
  is $\nat, {(p_m)}_{m \in \nat}$ where $\nat$ has the discrete
  topology and each $p_m$ is again the identity map.  Each $X_n$ is
  compact, in fact Noetherian (and $T_1$, too), but $\nat$ is not
  compact.
\end{remark}

We also need the following technical lemma.
\begin{lemma}
  \label{lemma:QC:sober}
  Let $X$ be a sober space, $A$ be a saturated subset of $X$ and $C$
  be a closed subset of $X$.  Then $A \cap C$, qua subspace, is sober.
\end{lemma}
\begin{proof}
  We recall that the specialization preordering of $A \cap C$
  coincides with restriction of that of $X$ to $A \cap C$.  In
  particular, there is no ambiguity as to our use of $\leq$.  Also,
  $A \cap C$ is then $T_0$.

  Let $D$ be a closed subset of $A \cap C$, with the subspace
  topology.  Consider the closure $cl (D)$ of $D$ in $X$.  We claim
  that $cl (D) \cap A \cap C = D$.  The inclusion $D \subseteq cl (D)
  \cap A \cap C$ is clear.  In the converse direction, since $D$ is
  closed in the subspace topology, there is a closed subset $C'$ of
  $X$ such that $D = C' \cap A \cap C$.  In particular, $C'$ is closed
  and contains $D$, hence it contains $cl (D)$.  Therefore $D = C'
  \cap A \cap C$ contains $cl (D) \cap A \cap C$.

  Now assume that $D$ is irreducible in $A \cap C$.  We claim that
  $cl (D)$ is irreducible in $X$.  If $cl (D)$ is included in the
  union $C' \cup C''$ of two closed subsets of $X$, then $D$ is
  included in the union of the closed subsets $C' \cap D$ and
  $C'' \cap D$ of $D$, hence is included in one of them, say $C'$.
  Then $D \subseteq C'$, so $cl (D) \subseteq C'$.

  Since $X$ is sober, there is a point $x \in X$ such that $cl
  (D) = \dc x$.  Hence $D = cl (D) \cap A \cap C = \dc x \cap A \cap
  C$.  Pick any point $y$ in $D$.  Then $y$ is in $cl (D) = \dc x$, so
  $y \leq x$.  Since $y$ is in $D \subseteq A$ and $A$ is
  upwards-closed, $x$ is in $A$, too.  Since $D$ is included in $C$
  and $C$ is closed in $X$, $cl (D)  = \dc x$ is also included in $C$,
  so $x$ is in $C$.  Since $x$ is in $A \cap C$, and since $x \in cl
  (D)$, $x$ is in $cl (D) \cap A \cap C = D$.  Moreover, $x$ is
  larger than any point of $cl (D)$, hence of $D$.   It follows that
  $x$ is the largest element of $D$ in $A \cap C$, hence that $D$ is
  the closure of $\{x\}$ in $A \cap C$.
\end{proof}

We shall also require the following, which is the sober non-$T_1$
analogue of Lemma~2.2 in \cite{Steenrod:univ:homol}.  We need to
observe that, given a projective system ${(p_{ij} \colon Q_j \to
  Q_i)}_{i \sqsubseteq j \in I}$, replacing $I$ by a cofinal subset
yields a subsystem with the same projective limit, up to isomorphism,
a fact that we shall reuse later on.
\begin{lemma}
  \label{lemma:steenrod:open}
  Let $Q, {(p_i)}_{i \in I}$ be the canonical projective limit of a
  projective system
  ${(p_{ij} \colon Q_j \to Q_i)}_{i \sqsubseteq j \in I}$ of compact
  sober spaces.  For every $i \in I$, for every open neighborhood $U$
  of $\upc p_i [Q]$ in $Q_i$, there is an index $j \in I$ such that $i
  \sqsubseteq j$ and $\upc p_{ij} [Q_j] \subseteq U$.
\end{lemma}
\begin{proof}
  Replacing $I$ by the cofinal family $\upc i$, we may assume that $i$
  is the least element of $I$.

  Assume the result was wrong.  For every $j \in \upc i$, $Q_j$ is not
  included in the open set $p_{ij}^{-1} (U)$.  Therefore
  $C_j \eqdef Q_j \diff p_{ij}^{-1} (U)$ is a non-empty closed subset
  of $Q_j$.  In particular, it is compact, and sober by
  Lemma~\ref{lemma:QC:sober}.  For all $j \sqsubseteq k$ in $\upc i$,
  $p_{jk}$ restricts to a map from $C_k$ to $C_j$: for every
  $x \in C_k$, if $p_{jk} (x)$ were not in $C_j$, then it would be in
  $p_{ij}^{-1} (U)$, so $p_{ij} (p_{jk} (x)) = p_{ik} (x)$ would be in
  $U$, hence $x$ would be in $p_{ik}^{-1} (U)$, contradicting the fact
  that $x$ is in $C_k$.  This way, we obtain a projective system
  ${(p_{jk} \colon C_k \to C_j)}_{j \sqsubseteq k \in \upc i}$ of
  non-empty compact sober spaces.  By Theorem~\ref{thm:steenrod}, its
  canonical projective limit $C$ is non-empty.  Pick an element
  $\vec x \eqdef {(x_j)}_{j \in \upc i}$ from $C$.  Since
  $p_{jk} (x_k) = x_k$ for all $j \sqsubseteq k$ in $\upc i$, $\vec x$
  is in $Q$.  Then $x_i = p_i (x)$ is in $\upc p_i [Q] \subseteq U$,
  which is impossible since $x_i$ is in $C_i = Q_i \diff U$ by
  construction.
\end{proof}

\section{Projective systems of locally compact sober spaces}
\label{sec:proj-syst-core}

We can now apply our Prohorov-like theorem on locally compact sober
spaces.  Steenrod's theorem will cater for uniform tightness.

\begin{theorem}[Projective limits of continuous valuations, locally
  compact sober case]
  \label{thm:prohorov:loccomp}
  Let
  ${(p_{ij} \colon (X_j, \nu_j) \to (X_i, \nu_i))}_{i \sqsubseteq j
    \in I}$ be a projective system of valued spaces.  Let
  $X, {(p_i)}_{i \in I}$ be a projective limit of the underlying
  projective system
  ${(p_{ij} \colon X_j \to X_i)}_{i \sqsubseteq j \in I}$.

  If $I$ has a countable cofinal subset,
  and if every $X_i$ is locally compact and sober, then
  there is a unique continuous valuation $\nu$ on $X$ such that for
  every $i \in I$, $\nu_i = p_i [\nu]$.  Moreover, $\nu$ is tight.
\end{theorem}
\begin{proof}
  It is easy to see that, since $I$ has a countable cofinal subset, it
  also has a cofinal monotone sequence
  $i_0 \sqsubseteq i_0 \sqsubseteq \cdots \sqsubseteq i_n \sqsubseteq
  \cdots$.  Replacing $I$ by that sequence, we may assume that
  $I = \nat$ and $\sqsubseteq$ is the usual ordering on $\nat$.  We
  are given a projective system
  ${(p_{mn} \colon (X_n, \nu_n) \to (X_m, \nu_m))}_{m \leq n \in
    \nat}$ of valued spaces, and we wish to find a (tight) continuous
  valuation $\nu$ on $X$ such that for every $n \in \nat$,
  $\nu_{i_n} = p_{i_n} [\nu]$.

  We verify uniform tightness.  Fix $m \in \nat$, let $U_m$ be an open
  subset of $X_m$, and $r \ll \nu_m (U_m)$.  Using local compactness,
  $U_m$ is the directed union of the interiors of compact saturated
  subsets of $U_m$.  Since $r \ll \nu_m (U_m)$ and $\nu_m$ is
  Scott-continuous, there is a compact saturated subset $Q_m$ of $U_m$
  such that $r \ll \nu_m (\interior {Q_m})$.  Let
  $U_{m+1} \eqdef p_{m(m+1)}^{-1} (\interior {Q_m})$.  We have
  $r \ll \nu_m (\interior {Q_m}) = p_{m(m+1)} [\nu_{m+1}] (\interior
  {Q_m}) = \nu_{m+1} (U_{m+1})$, so, by the same argument, there is a
  compact saturated subset $Q_{m+1}$ of $U_{m+1}$ such that
  $r \ll \nu_{m+1} (\interior {Q_{m+1}})$.  Iterating the process, we
  build compact saturated subsets $Q_n$ of
  $U_n \eqdef p_{(n-1)n}^{-1} (\interior {Q_{n-1}})$ such that
  $r \ll \nu_n (\interior {Q_n})$ for every $n > m$.

  When $n < m$, define $Q_n$ as $\upc p_{nm} [Q_m]$.  All the sets
  $Q_n$ are compact saturated in $X_n$, and $p_{(n-1)n}$ maps $Q_n$ to
  $Q_{n-1}$ for every $n \geq 1$.  Seeing $Q_n$ as a subspace instead
  of a mere subset of $X_n$, $Q_n$ is compact, and also sober by
  Lemma~\ref{lemma:QC:sober} (with $C \eqdef \emptyset$), since $Q_n$
  is saturated and $X_n$ is sober.  Steenrod's
  theorem~\ref{thm:steenrod} tells us that the canonical projective
  limit $Q$ of ${(p_{mn} \colon Q_n \to Q_m))}_{m \leq n \in \nat}$ is
  a compact space.  By construction, $Q$ is a subspace of the
  canonical projective limit $X, {(p_n)}_{n \in \nat}$ of
  ${(p_{mn} \colon Q_n \to Q_m))}_{m \leq n \in \nat}$.  Being compact
  as a subspace, it is also compact as a subset of
  $X$.

  The specialization preordering on $X$ is the restriction of that on
  $\prod_{n \in \nat} X_n$, hence is the componentwise preordering.
  Since every $Q_n$ is saturated (upwards-closed), one checks easily
  that $Q$ is also upwards-closed.  Since $p_m$ is a map from $Q$ to
  $Q_m$, $p_m [Q]$ is included in $Q_m$ hence in $U_m$, so
  $\upc p_m [Q] \subseteq U_m$.

  We now claim that $r \leq \nu_n^\bullet (\upc p_n [Q])$ for every
  $n \geq m$.  (We do not need to check that for $n <m$, by
  Remark~\ref{rem:unif:tight:kripke}.)  Pick an arbitrary open
  neighborhood $V$ of $\upc p_n [Q]$.  In particular,
  $Q \subseteq p_n^{-1} (V)$.  Since $\upc p_n [Q] \subseteq V$ and
  $p_n$ maps $Q$ to $Q_n$, $\upc p_n [Q]$ is included in the open
  subset $V \cap Q_n$ of $Q_n$.  By Lemma~\ref{lemma:steenrod:open},
  there is an index $k \geq n$ such that
  $\upc p_{nk} [Q_k] \subseteq V \cap Q_n$, hence such that
  $Q_k \subseteq p_{nk}^{-1} (V)$.  Recall that
  $r \ll \nu_k (\interior {Q_k})$, so
  $r \leq \nu_k (\interior {Q_k}) \leq \nu_k (p_{nk}^{-1} (V)) =
  p_{nk} [\nu_k] (V) = \nu_n (V)$.

  Hence ${(\nu_n)}_{n \in \nat}$ is uniformly tight.  We conclude by
  Theorem~\ref{thm:prohorov}.
\end{proof}

\begin{remark}
  \label{rem:prohorov:loccomp:sober}
  Sobriety is essential in Theorem~\ref{thm:prohorov:loccomp}.
  Consider again the projective system of
  Remark~\ref{rem:steenrod:sober}.  For each $m \in \nat$, define
  $\nu_m$ on $X_m$ by $\nu_m (U) \eqdef 1$ if
  $U \cap \upc m \neq \emptyset$, $0$ otherwise.  One checks easily
  that $\nu_m$ is a continuous valuation.  However, there is no
  continuous valuation $\nu$ on $\nat$ such that $\nu_m = p_m [\nu]$
  for every $m \in \nat$.  Indeed, such a $\nu$ would necessarily map
  every finite subset $U$ of $\nat$ to $0$, and every cofinite subset
  to $1$, contradicting Scott-continuity.
\end{remark}

\section{Projective limits of complete metric spaces, and more}
\label{sec:proj-limits-compl}

A similar technique applies to the case of complete metric spaces, and
more generally, of $G_\delta$ subsets of locally compact sober spaces.

All complete metric spaces (with their open ball topology), and more
generally all continuous complete quasi-metric spaces $X, d$ (with
their $d$-Scott topology), arise as $G_\delta$ subsets of their space
$\mathbf B (X, d)$ of formal balls, and the latter is a continuous
dcpo, hence is locally compact and sober.  That follows from
\cite{edalat98} for complete metric spaces.  For the larger class of
continuous complete quasi-metric spaces, this follows from
Proposition~2.6 of \cite{GLN-lmcs17}, which states that $X$ occurs as
a $G_\delta$ subset of $\mathbf B (X, d)$ whenever $X, d$ is standard;
all continuous quasi-metric spaces are standard, and $X, d$ is
continuous complete if and only if $\mathbf B (X, d)$ is a continuous
dcpo (Theorem~3.7 of loc.\ cit.).

Our final theorem will not apply to all continuous valuations, rather
to locally finite continuous valuations.  A valuation $\nu$ is
\emph{locally finite} if and only if every point $x$ has an open
neighborhood $U$ such that $\nu (U) < \infty$.  This is a standard
notion in measure theory.
\begin{lemma}
  \label{lemma:locfin:val}
  Let $\nu$ be a valuation on a topological space $X$.  The following
  are equivalent:
  \begin{enumerate}
  \item $\nu$ is locally finite;
  \item $X$ is a union of open subsets $V$ such that $\nu (V) <
    \infty$;
  \item every open subset $U$ of $X$ is a union of open subsets $V$
    such that $\nu (V) < \infty$;
  \item every open subset $U$ of $X$ is a directed union of open
    subsets $V$ such that $\nu (V) < \infty$.
  \end{enumerate}
\end{lemma}
\begin{proof}
  $1 \limp 2$.  For each point $x$ of $X$, pick an open subset $U_x$
  of $X$ such that $\nu (U_x) < \infty$.  Then $X = \bigcup_{x \in X}
  U_x$.

 $2 \limp 3$.  Assume that $X = \bigcup_{i \in I} U_i$ where $\nu
 (U_i) < \infty$ for each $i \in I$.  For every open subset $U$ of
 $X$, $U$ is equal to $\bigcup_{i \in I} U_i \cap U$, and $\nu (U_i
 \cap U) \leq \nu (U_i) < \infty$ for every $i \in I$.

 $3 \limp 4$.  Let $U$ be an open subset of $X$ and assume that
 $U = \bigcup_{i \in I} U_i$ where $\nu (U_i) < \infty$ for each
 $i \in I$.  Then $U$ is also equal to the directed union of the sets
 $\bigcup_{i \in J} U_i$, $J \in \Pfin (I)$, and
 $\nu (\bigcup_{i \in J} U_i) \leq \sum_{i \in J} \nu (U_i) < \infty$.

  $4 \limp 1$.  Take $U \eqdef X$, and write it as $\dcup_{i \in I}
  \nu (V_i)$ where $\nu (V_i) < \infty$ for every $i \in I$.  For
  every $x \in X$, $x$ is in some $V_i$.
\end{proof}

The following is standard, and the proof is left to the reader.
\begin{lemma}
  \label{lemma:Gdelta:subspace}
  Let $X$ be a subspace of a topological space $Y$.  For every subset
  $E$ of $X$,
  \begin{enumerate}
  \item $E$ is compact in $X$ if and only if $E$ is compact in $Y$;
  \item if $X$ is upwards-closed in $Y$, then $E$ is saturated in $X$
    if and only if $E$ is saturated in $Y$;
  \item if $X$ is a $G_\delta$ subset of $Y$, then $E$ is $G_\delta$
    in $X$ if and only if $E$ is $G_\delta$ in $Y$.  \qed
  \end{enumerate}
\end{lemma}

Every locally finite continuous valuation $\nu$ on a locally compact
sober space $X$ extends to a measure on the Borel $\sigma$-algebra
$\Borel X$.  Moreover, if $\nu$ is bounded, then the extension is
unique.  This was proved by Alvarez-Manilla
\cite[Theorem~3.27]{alvarez00}, see also
\cite[Theorem~5.3]{KL:measureext}.  The same result had been proved
earlier by Lawson \cite[Corollary~3.5]{Lawson:valuation} under an
additional second countability assumption, and for bounded continuous
valuations.

Conversely, the measures that restrict to a continuous valuation on
$\Open X$ are usually called \emph{$\tau$-smooth} in measure theory.
\begin{lemma}
  \label{lemma:Gdelta:comp}
  Let $Y$ be a locally compact sober space, and $\mu$ be a bounded
  $\tau$-smooth measure on $(Y, \Borel Y)$.  For every $G_\delta$
  subset $E$ of $Y$, for every $\epsilon > 0$, there is a compact
  $G_\delta$ subset $Q$ of $Y$ included in $E$ such that
  $\mu (Q) > \mu (E) - \epsilon$.
\end{lemma}
\begin{proof}
  Write $E$ as the intersection $\fcap_{n \in \nat} U_n$ of an
  antitone sequence of open subsets $U_n$, $n \in \nat$.  Since $Y$ is
  locally compact, $U_0$ is the directed union of sets of the form
  $\interior {Q_0}$, where $Q_0$ is compact saturated and included in
  $U_0$.  Since $\mu$ is $\tau$-smooth,
  $\mu (\interior {Q_0}) > \mu (U_0) - \epsilon / 2$ for some compact
  saturated subset $Q_0$ of $U_0$.  Call that inequality $(a_0)$.

  This is the base case of a construction of compact saturated subsets
  $Q_n$ by induction on $n$.  Assume that we have built $Q_n$ as a
  compact saturated subset of $Y$ included in $U_n$.  By the same
  argument, there is a compact saturated subset $Q_{n+1}$ of
  $U_{n+1} \cap \interior {Q_n}$ such that
  $\mu (\interior {Q_{n+1}}) > \mu (U_{n+1} \cap \interior {Q_n}) -
  \epsilon/2^{n+1}$.  Using the modularity law,
  $\mu (U_{n+1} \cap \interior {Q_n}) = \mu (U_{n+1}) + \mu (\interior
  {Q_n}) - \mu (U_{n+1} \cup \interior {Q_n})$.  Since
  $U_{n+1} \subseteq U_n$ and $Q_n \subseteq U_n$, that is larger than
  or equal to $\mu (U_{n+1}) + \mu (\interior {Q_n}) - \mu (U_n)$.
  Hence
  $\mu (\interior {Q_{n+1}}) > \mu (\interior {Q_n}) + \mu (U_{n+1}) -
  \mu (U_n) - \epsilon/2^{n+1}$.  Call that inequality $(a_{n+1})$.

  We build a telescoping sum out of the inequalities $(a_0)$, $(a_1)$,
  \ldots, $(a_n)$, and we obtain that
  $\mu (\interior {Q_n}) > \mu (U_n) - \epsilon (1-1/2^n) > \mu (U_n)
  - \epsilon$.

  Let $Q \eqdef \fcap_{n \in \nat} Q_n$.  Since $Y$ is sober hence
  well-filtered, $Q$ is compact saturated.  $Q$ is also equal to
  $\fcap_{n \in \nat} \interior {Q_n}$, since
  $Q_{n+1} \subseteq \interior {Q_n}$ for every $n \in \nat$.  Hence
  $Q$ is also a $G_\delta$ subset of $Y$.  By a familiar equality due
  to Kolmogorov, which applies to every bounded measure $\mu$,
  $\mu (Q) = \finf_{n \in \nat} \mu (\interior {Q_n})$.  That is
  larger than or equal to
  $\finf_{n \in \nat} \mu (U_n) - \epsilon = \mu (E) - \epsilon$.
  Finally, $Q$ is included in $E$ since $Q_n$ is included in $U_n$ for
  every $n \in \nat$.
\end{proof}

Given a continuous valuation $\nu$ on a space $Y$, and an open subset
$V$ of $Y$, we define $\nu_{|V}$ by
$\nu_{|V} (U) \eqdef \nu (U \cap V)$.  This is again a continuous
valuation, and a bounded one if $\nu (U) < \infty$.

\begin{theorem}
  \label{thm:prohorov:loccomp:Gdelta}
  Let
  ${(p_{ij} \colon (X_j, \nu_j) \to (X_i, \nu_i))}_{i \sqsubseteq j
    \in I}$ be a projective system of valued spaces.  Let
  $X, {(p_i)}_{i \in I}$ be a projective limit of the underlying
  projective system
  ${(p_{ij} \colon X_j \to X_i)}_{i \sqsubseteq j \in I}$.

  If $I$ has a countable cofinal subset,
  if every $\nu_i$ is locally finite, and if every $X_i$ embeds as a
  $G_\delta$ subset in a locally compact sober space, then there is a
  unique continuous valuation $\nu$ on $X$ such that for every
  $i \in I$, $\nu_i = p_i [\nu]$.  Moreover, $\nu$ is tight.
\end{theorem}
\begin{proof}
  As in Theorem~\ref{thm:prohorov:loccomp}, the claim reduces to the
  case of a countably indexed system
  ${(p_{mn} \colon (X_n, \nu_n) \to (X_m, \nu_m))}_{m \leq n \in
    \nat}$ of valued spaces.

  By assumption, each $X_n$ embeds as a $G_\delta$ subset of some
  locally compact sober space $Y_n$.  Hence, up to isomorphism, there
  is an antitone sequence ${(V_{nk})}_{k \in \nat}$ of open subsets of
  $Y_n$ such that $X_n = \fcap_{k \in \nat} V_{nk}$, for each
  $n \in \nat$.  Then $\nu_n$ extends to a continuous valuation on the
  whole of $Y_n$, which we write as $\nu^*_n$ (instead of
  $i_n [\nu_n]$, where $i_n$ is the embedding of $X_n$ into $Y_n$),
  and which is supported on $X_n$ (see Proposition~\ref{prop:supp}).
  Since $X_n$ embeds in $Y_n$ as a topological space, every open
  subset $U$ of $X_n$ can be written as $\widehat U \cap X_n$ for some
  open subset $\widehat U$ of $Y_n$.

  Let us show uniform tightness.  Fix $m \in \nat$, an open subset
  $U_m$ of $X_m$, and let $r \ll \nu_m (U_m)$.  If $r=0$, we can take
  $Q \eqdef \emptyset$, then the uniform tightness requirements
  $\upc p_m [Q] \subseteq U_m$ and
  $r \leq \nu_n^\bullet (\upc p_n (Q))$ for every $n \in \nat$ are
  trivial.  Henceforth, let us assume $r \neq 0$, so that
  $r < \nu_m (U_m)$.

  The local finiteness assumption allows us to assume that
  $\nu_m (U_m) < \infty$.  Indeed, $U_m$ is a directed union of open
  subsets $V$ such that $\nu_m (V) < \infty$.  Since $\nu_m$ is
  Scott-continuous, we must have $r < \nu_m (V)$ for one of them, and
  we replace $U_m$ with that $V$ if necessary.

  For each $n \geq m$, $\nu^*_{n|\widehat {p_{mn}^{-1} (U_m)}}$ is a
  bounded continuous valuation: boundedness is because
  $\nu^*_{n|\widehat {p_{mn}^{-1} (U_m)}} (Y_n) = \nu^*_n (\widehat
  {p_{mn}^{-1} (U_m)}) = \nu_n (p_{mn}^{-1} (U_m)) = p_{mn} [\nu_n]
  (U_m) = \nu_m (U_m) < \infty$.

  Let $\mu_n$ be the unique (bounded) measure that extends
  $\nu^*_{n|\widehat {p_{mn}^{-1} (U_m)}}$ to $(Y_n, \Borel {Y_n})$.
  For every $G_\delta$ subset $E$ of $X_n$, $E$ is also $G_\delta$ in
  $Y_n$ by Lemma~\ref{lemma:Gdelta:subspace}, item~3.  Hence
  $\mu_n (E)$ makes sense.  This is in particular true when $E$ is an
  open subset $U$ of $X_n$.

  We claim that: $(*)$ $\mu_n (U) = \nu_n (U \cap p_{mn}^{-1} (U_m))$
  for every open subset $U$ of $X_n$, for every $n \geq m$.  Indeed,
  $\mu_n (U) = \mu_n (\widehat U \cap X_n) = \finf_{k \in \nat} \mu_n
  (\widehat U \cap V_{nk})$ (remember that $\mu_n$ is bounded)
  $= \finf_{k \in \nat} \nu^*_{n|\widehat {p_{mn}^{-1} (U_m)}}
  (\widehat U \cap V_{nk}) = \finf_{k \in \nat} \nu^*_n (\widehat U
  \cap V_{nk} \cap \widehat {p_{mn}^{-1} (U_m)}) = \finf_{k \in \nat}
  \nu_n (U \cap p_{mn}^{-1} (U_m)) = \nu_n (U \cap p_{mn}^{-1}
  (U_m))$.

  It follows that: $(**)$ for every $G_\delta$ subset $E$ of $X_n$,
  for every $n \geq m$, $\mu_n (E) = \mu_{n+1} (p_{n(n+1)}^{-1} (E))$.
  When $E$ is an open set $U$, by using $(*)$ once in the first
  equality and once in the last equality,
  $\mu_n (U) = \nu_n (U \cap p_{mn}^{-1} (U_m)) = p_{n(n+1)}
  [\nu_{n+1}] (U \cap p_{mn}^{-1} (U_m)) = \nu_{n+1} (p_{n(n+1)}^{-1}
  (U) \cap p_{n(n+1)}^{-1} (p_{mn}^{-1} (U_m))) = \nu_{n+1}
  (p_{n(n+1)}^{-1} (U) \cap p_{m(n+1)}^{-1} (U_m)) = \mu_{n+1}
  (p_{n(n+1)}^{-1} (U))$.  When $E$ is the intersection
  $\fcap_{k \in \nat} V_k$ of an antitone sequence of open subsets
  $V_k$ of $X_n$, $\mu_n (E) = \finf_{k \in \nat} \mu_n (V_k)$
  (remember that $\mu_n$ is bounded)
  $= \finf_{k \in \nat} \mu_{n+1} (p_{n(n+1)}^{-1} (V_k)) = \mu_{n+1}
  (\fcap_{k \in \nat} p_{n(n+1)}^{-1} (V_k)) = \mu_{n+1}
  (p_{n(n+1)}^{-1} (E))$.

  We can now start our construction.  By $(*)$ with $U \eqdef U_m$ and
  $n \eqdef m$, $\mu_m (U_m) = \nu_m (U_m)$.  That is strictly larger
  than $r$ by assumption.  Let us pick $\epsilon > 0$ so that
  $\mu_m (U_m) > r + \epsilon$.

  Let us use Lemma~\ref{lemma:Gdelta:comp}.  Note that $\mu_m$ is
  $\tau$-smooth by definition.  Hence there is a compact $G_\delta$
  subset $Q_m$ of $Y_m$ included in $U_m$ such that
  $\mu_m (Q_m) > \mu_m (U_m) - \epsilon/2 > r + \epsilon/2$.

  Since $U_m$ is included in $X_m$, so is $Q_m$.  Since $X_m$ is a
  $G_\delta$ subset of $Y_m$, by Lemma~\ref{lemma:Gdelta:subspace},
  $Q_m$ is also a compact $G_\delta$ subset of $X_m$.

  Let $E_{m+1} \eqdef p_{m(m+1)}^{-1} (Q_m)$, by which we mean that
  $E_{m+1}$ is the set of points $x$ of $X_{m+1}$ (not $Y_{m+1}$,
  which would be meaningless) such that $p_{m(m+1)} (x) \in Q_m$.
  Since $p_{m(m+1)}$ is continuous, $E_{m+1}$ is a $G_\delta$ subset
  of $X_{m+1}$, hence also of 
  $Y_{m+1}$ by Lemma~\ref{lemma:Gdelta:subspace}, item~3.  We use
  Lemma~\ref{lemma:Gdelta:comp} again: there is a compact $G_\delta$
  subset $Q_{m+1}$ of $Y_{m+1}$ included in $E_{m+1}$ such that
  $\mu_{m+1} (Q_{m+1}) > \mu_{m+1} (E_{m+1}) - \epsilon/4$.  By
  $(**)$,
  $\mu_{m+1} (E_{m+1}) = \mu_{m+1} (p_{m(m+1)}^{-1} (Q_m)) = \mu_m
  (Q_m)$.  Therefore
  $\mu_{m+1} (Q_{m+1}) > \mu_m (Q_m) - \epsilon/4 > r+\epsilon/4$.
  Note also that $p_{m(m+1)} [Q_{m+1}] \subseteq Q_m$: for every
  $x \in Q_{m+1}$, $x$ is in $E_{m+1}$, so $p_{m(m+1)} (x) \in Q_m$.

  We iterate the process and build compact $G_\delta$ subsets $Q_n$ of
  $X_n$ for each $n \geq m$ such that
  $\mu_n (Q_n) > r + \epsilon/2^{n+1}$ and
  $p_{n(n+1)} [Q_{n+1}] \subseteq Q_n$.

  We claim that the inequality $\mu_n (Q_n) > r + \epsilon/2^{n+1}$
  implies $\nu_n^\bullet (Q_n) > r$, too.  For every open subset $U$
  of $X_n$ that contains $Q_n$, $\widehat U$ also contains $Q_n$, so
  $\nu_n (U) = \nu^*_n (\widehat U) \geq \nu^*_{n|p_{mn}^{-1} (U_m)}
  (\widehat U) = \mu_n (\widehat U) \geq \mu_n (Q_n) > r$.

  The rest of the argument is as for
  Theorem~\ref{thm:prohorov:loccomp}.  When $n < m$, define $Q_n$ as
  $\upc p_{nm} [Q_m]$.  All the sets $Q_n$ are compact saturated in
  $X_n$, and $p_{n(n+1)}$ maps $Q_{n+1}$ to $Q_n$ for every
  $n \in \nat$.  By Steenrod's theorem~\ref{thm:steenrod}, the
  canonical projective limit $Q$ of
  ${(p_{mn} \colon Q_n \to Q_m)}_{m \leq n \in \nat}$ is compact.  We
  verify that $\nu_n^\bullet (\upc p_n [Q]) \geq r$ for every
  $n \geq m$.  For every open neighborhood $V$ of $\upc p_n [Q]$,
  $\upc p_n [Q]$ is included in the open subset $V \cap Q_n$ of the
  compact space $Q_n$.  By Lemma~\ref{lemma:steenrod:open}, there is
  an index $k \geq n$ such that
  $\upc p_{nk} [Q_k] \subseteq V \cap Q_n$, hence such that
  $Q_k \subseteq p_{nk}^{-1} (V)$.  Since $r < \nu_k^\bullet (Q_k)$,
  it follows that $r < \nu_k (p_{nk}^{-1} (V)) = \nu_n (V)$.  This
  shows that $\nu_n^\bullet (\upc p_n [Q]) > r$.  We can now apply
  Remark~\ref{rem:unif:tight:kripke}, so ${(\nu_n)}_{n \in \nat}$ is
  uniformly tight.  We conclude by Theorem~\ref{thm:prohorov}.
\end{proof}

\begin{corollary}[Projective limits of continuous valuations,
  continuous complete quasi-metric case]
  \label{corl:prohorov:metric}
  Let
  ${(p_{ij} \colon (X_j, \nu_j) \to (X_i, \nu_i))}_{i \sqsubseteq j
    \in I}$ be a projective system of valued spaces.  Let
  $X, {(p_i)}_{i \in I}$ be a projective limit of the underlying
  projective system
  ${(p_{ij} \colon X_j \to X_i)}_{i \sqsubseteq j \in I}$.

  If $I$ has a countable cofinal subset,
  if every $\nu_i$ is locally finite, and if every $X_i$ is a
  continuous complete quasi-metric space (e.g., a complete metric
  space, a Polish space), then there is a unique continuous valuation
  $\nu$ on $X$ such that for every $i \in I$, $\nu_i = p_i [\nu]$.
  Moreover, $\nu$ is tight.  \qed
\end{corollary}



\bibliographystyle{alpha}
\newcommand{\etalchar}[1]{$^{#1}$}



\end{document}